\documentclass[11pt]{amsart}
\usepackage{amssymb, amstext, amscd, amsmath}
\usepackage{mathtools, xypic, paralist, color, rotating}
\usepackage{verbatim}
\usepackage{enumitem}
\usepackage{float}
\usepackage{setspace}
\usepackage{pifont}

\usepackage{tikz}
\usetikzlibrary{fit,positioning,arrows,automata,calc}
\tikzset{
  main/.style={circle, minimum size = 30pt, thick, draw = black!80, node distance = 10mm},
  connect/.style={-latex, thick},
  box/.style={rectangle, draw = white!100}
}
\usepackage{tikz-cd}

\floatstyle{boxed} 
\restylefloat{figure}

\numberwithin{equation}{section}
\usepackage[a4paper]{geometry}
\usepackage{quoting}
\quotingsetup{vskip=.1in}
\quotingsetup{leftmargin=.6in} 
\quotingsetup{rightmargin=.6in}

\let\OLDthebibliography\thebibliography
\renewcommand\thebibliography[1]{
  \OLDthebibliography{#1}
  \setlength{\parskip}{0pt}
  \setlength{\itemsep}{2pt plus 0.5ex}
}

\makeatletter
\def\@cite#1#2{{\m@th\upshape\bfseries%
[{#1\if@tempswa{\m@th\upshape\mdseries, #2}\fi}]}}
\makeatother
\theoremstyle{plain}
\newtheorem{theorem}{Theorem}[section]
\newtheorem{corollary}[theorem]{Corollary}
\newtheorem{proposition}[theorem]{Proposition}
\newtheorem{lemma}[theorem]{Lemma}
\newtheorem{thmintro}{Theorem}

\newtheorem{corintro}[thmintro]{Corollary}
\theoremstyle{definition}
\newtheorem{definition}[theorem]{Definition}
\newtheorem{example}[theorem]{Example}

\newtheorem{remark}[theorem]{Remark}

\newtheorem{question}{Question}

\theoremstyle{remark}

\mathtoolsset{centercolon}
  \newcommand{\A}{{\mathcal{A}}}
  \newcommand{\B}{{\mathcal{B}}}
  \newcommand{\C}{{\mathcal{C}}}
  \newcommand{\D}{{\mathcal{D}}}
  \newcommand{\E}{{\mathcal{E}}}
  \newcommand{\F}{{\mathcal{F}}}
  \newcommand{\G}{{\mathcal{G}}}

  \newcommand{\I}{{\mathcal{I}}}
  \newcommand{\J}{{\mathcal{J}}}
  \newcommand{\K}{{\mathcal{K}}}
\renewcommand{\L}{{\mathcal{L}}}

\renewcommand{\P}{{\mathcal{P}}}

\renewcommand{\S}{{\mathcal{S}}}
  \newcommand{\T}{{\mathcal{T}}}

  \newcommand{\X}{{\mathcal{X}}}
  \newcommand{\Y}{{\mathcal{Y}}}
  
\newcommand{\eps}{\varepsilon}
\def\al{\alpha}
\def\be{\beta}

\def\de{\delta}
\def\ze{\zeta}
\def\io{\iota}
\def\ka{\kappa}
\def\la{\lambda}

\def\om{\omega}

\def\si{\sigma}

\newcommand\vphi{\varphi}

\newcommand{\bC}{\mathbb{C}}

\newcommand{\bF}{\mathbb{F}}

\newcommand{\bN}{\mathbb{N}}

\newcommand{\bZ}{\mathbb{Z}}
\newcommand{\bR}{\mathbb{R}}

\newcommand{\Bx}{{\mathbf{x}}}
\newcommand{\By}{{\mathbf{y}}}
\newcommand{\Bz}{{\mathbf{z}}}
\newcommand{\AND}{\text{ and }}

\newcommand{\OR}{\text{ or }}
\newcommand{\foral}{\text{ for all }}
\newcommand{\qand}{\quad\text{and}\quad}

\newcommand{\qfor}{\quad\text{for}\quad}

\newcommand{\ca}{\mathrm{C}^*}

\newcommand{\wt}{\widetilde}
\newcommand{\wh}{\widehat}

\newcommand{\cov}{\operatorname{c}}

\newcommand{\ev}{\operatorname{ev}}
\newcommand{\fock}{\operatorname{F}}

\newcommand{\id}{{\operatorname{id}}}

\newcommand{\mt}{\emptyset}

\newcommand{\spn}{\operatorname{span}}

\newcommand{\sca}[1]{\left\langle#1\right\rangle} 
\newcommand{\bo}[1]{\mathbf{#1}} 

\newcommand{\tes}[7]{
	\xymatrix@C=2cm@R=1.5cm{
		K_0\left(#1\right) \ar[r]^{#2} & K_0\left(#3\right) \ar[r]^{#4} & K_0\left(#5\right) \ar[d] \\
		K_1\left(#5\right) \ar[u] & K_1\left(#3\right) \ar[l]^{#7} & K_1\left(#1\right) \ar[l]^{#6}
	}
}

\addtocontents{toc}{\protect\setcounter{tocdepth}{1}}

\begin{document}

\title[Finite-dimensional approximations for C*-algebras of product systems]{Finite-dimensional approximations for C*-algebras of product systems}

\author[E.T.A. Kakariadis]{Evgenios T.A. Kakariadis}
\address{Department of Mathematics\\ National and Kapodistrian University of Athens\\ Athens\\ 157 84\\ Greece}
\email{evkakariadis@math.uoa.gr}

\author[I.A. Paraskevas]{Ioannis Apollon Paraskevas}
\address{Department of Mathematics\\ National and Kapodistrian University of Athens\\ Athens\\ 157 84\\ Greece}
\email{ioparask@math.uoa.gr}

\thanks{2020 {\it Mathematics Subject Classification.} 46L05, 46L06, 46L08, 47L55}

\thanks{{\it Key words and phrases:} Product systems, Fock covariance, nuclearity, exactness.}

\begin{abstract}
We study C*-algebras generated by representations of product systems over unital subsemigroups of amenable discrete groups.
In the first part, we show that exactness of the generated C*-algebra is equivalent to exactness of the cores over the constructible ideals.
Nuclearity of the generated C*-algebra is equivalent to nuclear embeddability of the cores in the fixed point algebra.
In the second part, we further study injective equivariant Fock covariant representations, and whether the properties of the cores are inherited from properties of the coefficient C*-algebra, as in the case of compactly aligned saturated product systems over right LCM semigroups.
If the semigroup is finitely aligned and the product system is saturated and compactly aligned, then exactness of the cores is equivalent to exactness of the coefficient C*-algebra.
If in addition the semigroup is strongly finitely aligned, then nuclear embeddability of the cores is equivalent to nuclear embeddability of the coefficient C*-algebra.
These results do not hold without the compact alignment hypothesis, even for right LCM semigroups.
\end{abstract}

\maketitle

\section{Introduction}

\subsection{Framework}

The theory of C*-algebras accommodates a variety of ways for constructing an analytical object from geometric data, widely interpreted as actions on a (possibly non-commutative) state space.
For these models, the natural question is how properties of the data impact the structure of the generated C*-algebras, and vice versa.
A two-way passage allows us to construct examples and counterexamples to test hypotheses, by referring to the original structure and invoking combinatorial arguments, rather than repeating analytic arguments.
In this respect, graphs and semigroups have been at the core of this investigation since the 1970s.

In the current project we study the properties of exactness and nuclearity that sit at the core of the classification programme of C*-algebras, for a class of structures that encompasses numerous constructions and possibly irreversible transformations.
Our focal objects are product systems over a unital subsemigroup of a discrete group, and the analytical objects arise from representations that respect the Fock space model.
As these objects cover the case when the semigroup is just a group, the finite-dimensional approximation properties we study reside in the context of exact or amenable groups.

A (concrete) product system $X$ over a unital subsemigroup $P$ in a discrete group $G$ consists of a family $\{X_p\}_{p \in P}$ of norm-closed operator spaces in a common $\B(H)$ such that $X_{e_P}:=A$ is a C*-algebra, and
\[
X_p \cdot X_q \subseteq X_{pq}
\textup{ and } 
X_p^* \cdot X_{pq} \subseteq X_q
\textup{ for all }
p,q \in P,
\]
where we write $e_P = e_G$ for the common unit of $P$ and $G$.
The product system $X$ is called \emph{saturated} if in addition the closed linear span of $X_p \cdot X_q$ equals $X_{pq}$ for all $p,q\in P$.
From the definition it follows that each \emph{fibre} $X_p$ becomes a C*-correspondence over the \emph{coefficient algebra} $A$.
The case of $P = \bZ_+$ has been extensively studied, propelled by Katsura's Gauge-Invariant Uniqueness Theorems \cite{Kat04, Kat07}, and culminating in a result of Brown--Tikuisis--Zelenberg \cite{BTZ18} about their classifiability (a hopefully extensive but arguably non-exhaustive survey for $P = \bZ_+$ can be found in \cite{KP26}).
Moreover, product systems extend the notion of Fell bundles, which arise when $P = G$.
Product systems were introduced by Arveson \cite{Arv89} and later put in context by Dinh \cite{Din91} for discrete subsemigroups of $\bR_+$.
They have been under considerable study since then, with major advancements over the past decade, offering a unifying language for semigroup operator algebras (when $X_p = \bC$ for all $p \in P$) and operator algebras of higher-rank graphs and semigroup transformations.

The first quantisation of a product system is realised through left creation operators.
Given $X$ we construct the Fock space $\F X := \sum^\oplus_{p \in P} X_p$ as a right Hilbert $A$-module and we define the \emph{Fock representation} $\la \colon X \to \L(\F X)$ by
\[
\la_p(\xi_p) \eta_q := \xi_p \eta_q \in X_{pq},
\foral p,q \in P, \xi_p \in X_p, \eta_q \in X_q.
\]
The C*-algebra generated by $\{\la_p(X_p)\}_{p \in P}$ is denoted by $\T_\la(X)$ and is referred to, herein, as the \emph{Fock C*-algebra}.
The group embeddability of $P$ in $G$ defines a \emph{coaction} of $G$ on $\T_\la(X)$ satisfying
\[
\de_\la \colon \T_\la(X) \to \T_\la(X) \otimes \ca(G); \ \de_\la(\la_p(\xi_p)) = \la_p(\xi_p) \otimes u_p,
\]
and thus induces a Fell bundle structure
\[
\F\C X:=\{\F\C X_g\}_{g \in G}
\textup{ such that }
\F\C X_g := \{c \in \T_\la(X) \mid \de_\la(c) = c \otimes u_g\},
\]
identified in \cite{DKKLL22}.
As the coaction on $\T_\la(X)$ is normal, the Fock C*-algebra coincides with the reduced C*-algebra of $\F\C X$.

Modelled on $\la$, we say that a collection $t = \{t_p\colon X_p \to\B(H')\}_{p \in P}$ of linear maps is a \emph{representation} of $X$ if $t_{e_P}$ is a $*$-representation, and moreover
\[
t_p(\xi_p) t_q(\eta_q) = t_{pq}(\xi_p \eta_q)
\textup{ and }
t_p(\xi_p)^* t_{pq}(\eta_{pq}) = t_q(\xi_p^* \eta_{pq}),
\]
for all $p, q \in P$, $\xi_p \in X_p$, $\eta_q \in X_q$, $\eta_{pq} \in X_{pq}$.
We write $\T(X)$ for the universal C*-algebra with respect to representations of $X$, and $\T_{\cov}^{\fock}(X)$ for the universal C*-algebra of $\F\C X$. 
We say that a representation $t$ is \emph{Fock covariant} if the canonical map $t_\ast \colon \T(X) \to \ca(t)$ factors through $\T_{\cov}^{\fock}(X)$, or equivalently if $t$ promotes to a representation of the Fell bundle $\F\C X$.
The main question we wish to address is the following.

\begin{question}\label{Qn:main}
Under which conditions on $P$ and $X$ is exactness or nuclearity of the C*-algebra of a Fock covariant representation induced by exactness or nuclear embeddability of $A$?
\end{question}

The question refers to Fock covariant representations as it is known to fail for other representations, e.g., the universal C*-algebra of the representations of $X$ when $P = \bZ_+^2$ and $X_p = \bC$ for all $p \in P$ fails to be nuclear \cite{Mur96}.
Furthermore we consider \emph{injective} representations of $X$, i.e., when $t_{e_P}$ is injective, so that the original data from $X$ remain intact.
A further useful reduction is to ask for $t$ to be \emph{equivariant}, i.e., when the coaction of the \emph{universal Toeplitz C*-algebra} $\T(X)$ descends to $\ca(t)$, so that we can make use of the fixed point algebra $[\ca(t)]_{e_P}$.

\subsection{Reductions of the problem and compact alignment}

Question~\ref{Qn:main} has been considered in several contexts and settings.
For $P = \bZ_+$ this was fully resolved in \cite{Kat04}.
However, going beyond this point appears to require extra hypotheses on $X$.
The problem lies in how the building blocks of the fixed point algebra interact.

Such an obstruction to the study of the structure of Fock covariant representations was already observed by Fowler \cite{Fow02}.
In his work, product systems are introduced abstractly through a family 
\[
X = \left( \{X_p\}_{p \in P}, \{u_{p,q} \colon X_p \otimes_A X_q \to X_{pq} \}_{p, q \in P} \right)
\]
of C*-correspondences over a common C*-algebra $X_{e_P} := A$ and a family of unitaries $\{u_{p,q}\}_{p,q \in P}$ that respects the associativity of the stabilised tensor product triples.
This definition is equivalent to the concrete one modulo non-degeneracy, see \cite[Subsection 2.5]{KP25}.
Fowler introduced the notion of \emph{compact alignment} when $(G,P)$ is a quasi-lattice ordered group.
This notion borrows from the Nica covariance for semigroups \cite{Nic92}, and reflects finite alignment for higher-rank graphs introduced by Raeburn--Sims--Yeend \cite{RSY04}: the product system of a higher-rank graph is compactly aligned precisely when the graph is finitely aligned \cite{RS05}. It is automatically satisfied when the product system is \emph{proper}, i.e., when the left action on each fibre is by compacts.
The quasi-lattice ordered group assumption was further relaxed in a series of works, and eventually the setting of compact alignment was transferred by Kwa\'{s}niewski--Larsen \cite{KL19a, KL19b} to the class of right LCM semigroups.

One approach towards Question~\ref{Qn:main} has been by means of groupoid models.
Albandik--Meyer \cite{AM15} studied proper product systems over Ore monoids and obtained nuclearity of the boundary quotient from that of the coefficient algebra. 
Such a characterisation of nuclearity for the Fock C*-algebra of a compactly aligned product system over a quasi-lattice ordered group is also obtained by Rennie--Robertson--Sims \cite{RRS17}, modulo amenability of the constructed groupoid.
More recently, Hossain--Sundar \cite{HS26} used an $E_0$-semigroup and the Wiener--Hopf groupoid model of Muhly--Renault \cite{MR82} to obtain equivalence of exactness and nuclearity of the Fock C*-algebra (therein called the reduced C*-algebra) with that of the coefficient algebra when $P$ is an Ore monoid with an order unit, and under suitable hypotheses that allow the product system to be encoded as an injective dynamical system (see also \cite{Kak17} for $\bZ_+^d$-dynamical systems).

Another approach towards Question~\ref{Qn:main} has been through an analysis of the cores of the fixed point algebra, more closely related to Li's theory of semigroup C*-algebras \cite{Li12}.
For a semigroup $P$ we write $\J$ for the smallest family of right ideals of $P$ containing $\{\mt, P\}$ that is closed under left multiplication, taking pre-images under left multiplication, and finite intersections.
For a representation $t$ of a saturated product system $X$ and an $\Bx \in \J$ we write $\bo{K}_{\Bx, t_\ast}$ for the closed linear span of the sets of the form
\[
 t_{p_1}(X_{p_1})^* t_{q_1}(X_{q_1}) \cdots
t_{p_n}(X_{p_n})^* t_{q_n}(X_{q_n}) \text{ such that } \Bx = q_n^{-1} p_n \dots q_1^{-1} p_1 P, p_1^{-1} q_1 \cdots p_n^{-1} q_n = e_P.
\]
If $t$ inherits a $G$-coaction, then $[\ca(t)]_{e_P}$ arises as the inductive limit of finite sums of these $\bo{K}$-cores over finite $\cap$-closed subsets of $\J$, see \cite[Subsection 2.4]{KP25}.
Therefore, we can address Question~\ref{Qn:main} in two independent steps.

\begin{question}\label{Qn:main2}
Let $t$ be a Fock covariant representation of a product system $X$ over $P$.
\begin{enumerate}
\item Under which conditions on $P$ and $X$ is exactness or nuclearity of $\ca(t)$ equivalent to exactness or nuclear embeddability of the cores $\bo{K}_{\Bx, t_\ast}$ for all $\Bx \in \J$? 
\item Under which conditions on $P$ and $X$ is exactness or nuclear embeddability of the cores $\bo{K}_{\Bx, t_\ast}$ for all $\Bx \in \J$ equivalent to exactness or nuclear embeddability of the coefficient C*-algebra $A$? 
\end{enumerate}
\end{question}

This route was successfully taken in \cite{Kat04} for $P = \bZ_+$, with a two-out-of-three lemma for item (i) above.
Such a lemma was extended in \cite{Kak20} and employed therein for compactly aligned product systems over $\bZ_+^d$.
It was further used in \cite{KKLL23} to settle the case of saturated compactly aligned product systems over a right LCM semigroup $P$ that is controlled by a right LCM subsemigroup $\P$ of an amenable group $\G$.
In order to appreciate the strength of compact alignment, we recall that when $P$ is a right LCM semigroup then $\J = \{pP \mid p \in P\} \cup \{\mt\}$; if $X$ is in addition compactly aligned, then $\bo{K}_{pP, t_\ast} = [t_q(X_q) t_q(X_q)^*]$ for any $q \in P$ with $qP = pP$ and every Fock covariant representation $t$ (note here the independence from the writing of the principal ideal, see \cite{DKKLL22}).
Item (ii) of Question~\ref{Qn:main2} is then resolved in \cite{Kak20, KKLL23, Kat04} by employing a linking algebra argument from \cite{Kat04}.
A more general setting of controlled maps was used by an Huef--Nucinkis--Sehnem--Yang \cite{HNSY21} for nuclearity results of semigroup C*-algebras.

Katsoulis--Laca--Sehnem \cite{KLS25} further studied Question~\ref{Qn:main2} when the range of each left action contains all compacts, and an assumption on the support of each $\bo{K}_{\bo{x}, \la_\ast}$ is satisfied, see \cite[Equation (5.4)]{KLS25}.
These conditions yield exactness or nuclearity of each $\bo{K}_{\bo{x}, \la_\ast}$ from the coefficient algebra.
A two-out-of-three argument then gives nuclearity of $\T_\la(X)$.
Assuming further a linear independence condition on the $\bo{K}$-cores of $\la$, they obtain exactness of $\T_\la(X)$.
Their results apply when each $X_p$ is \emph{full}, i.e., $[X_p^*\cdot X_p] = X_{e_P}$, and $P$ is abelian, or the $ax+b$-monoid of an integral domain, or the Baumslag--Solitar monoid $\mathrm{BS}^+(m,n)$.

\subsection{Main results}

In the first part of this work we settle item (i) of Question~\ref{Qn:main2} by assuming just amenability of $G$; the saturation hypothesis on the product system is not required for this part.
In this case the canonical $*$-epimorphism $\T_{\cov}^{\fock}(X) \to \T_\la(X)$ is a $*$-isomorphism.

\begin{thmintro}[Theorem~\ref{T:Kcores} and Corollary~\ref{C:Excores-1}]
Let $P$ be a unital subsemigroup of an amenable discrete group $G$ and let $X$ be a product system over $P$.
Let $t$ be an equivariant representation of $X$.
Then $\ca(t)$ is exact if and only if $\bo{K}_{\Bx,t_\ast}$ is exact for every $\Bx \in \J$.
\end{thmintro}

This follows from an analysis of the cores as well as the connection of exactness with locally split (or locally liftable) extensions from the general C*-theory.
We break our arguments into several parts to make full use of each hypothesis.
In particular for the case of Fock covariant representations we have the following corollary.

\begin{corintro}[Corollary~\ref{C:Excores-1} and Corollary~\ref{C:Excores-2}]
Let $P$ be a unital subsemigroup of an amenable discrete group $G$ and let $X$ be a product system over $P$.
Then the following are equivalent:
\begin{enumerate}
\item $\T_\la(X)$ is exact.
\item $\ca(t)$ is exact for an injective Fock covariant representation $t$ of $X$. 
\item $\bo{K}_{\Bx, t_\ast}$ is exact for an injective Fock covariant representation $t$ of $X$, for all $\Bx \in \J$. 
\end{enumerate}
\end{corintro}

We separately consider the case where $G$ is exact.
We show that $\T_\la(X)$ is exact if and only if $\ca(t)$ is exact for a quotient of $\T_\la(X)$ that is injective on $A$, if and only if $\ca(t)$ is exact for an injective equivariant Fock covariant representation $t$ of $X$ with a normal coaction, see Corollary~\ref{C:Excores-3}.
The injectivity and Fock covariance hypotheses allow us to identify the $\bo{K}$-cores.
Thus we recover that, if $G$ is exact, then every quotient of the reduced C*-algebra of a semigroup is exact, without appealing to any dynamical system construction (see Corollary~\ref{C:exsem}).

We perform a similar study for nuclearity.
This in turn recovers that, if $G$ is amenable, then every quotient of the reduced C*-algebra of the semigroup (which in this case coincides with the Toeplitz C*-algebra of the semigroup in the sense of Laca--Sehnem \cite{LS22}) is nuclear \cite[Corollary 4.10]{LS22}.
This is done in Corollary~\ref{C:nucsem}, without appealing to any dynamical system construction.

\begin{thmintro}[Theorem~\ref{T:Nuccores}]
Let $P$ be a unital subsemigroup of an amenable discrete group $G$ and let $X$ be a product system over $P$.
Let $t$ be an equivariant representation of $X$.
Then $\ca(t)$ is nuclear if and only if $\bo{K}_{\Bx,t_\ast}$ is nuclearly embedded in $[\ca(t)]_{e_P}$ for every $\Bx \in \J$.
\end{thmintro}

These results settle item (i) of Question~\ref{Qn:main2} for any product system, and we next pass to the study of item (ii).
Unlike \cite{KLS25}, we seek conditions on the semigroup $P$ rather than further hypotheses on $X$, apart from compact alignment.
A natural candidate beyond right LCM semigroups is the class of finitely aligned semigroups, which has been under considerable study in the past decade originating from the work of Gould \cite{Gou92} and Spielberg \cite{Spi12, Spi14}, see for example \cite{BKQS18, CG21, DM14, ES18, LSV24, LV20, LV24, Li23, MS25}.
Such semigroups appear naturally even at the level of numerical subsemigroups of $\bZ_+$, see Example~\ref{Ex:fa}.

A semigroup $P$ is called \emph{finitely aligned} if for any $p,q \in P$ the intersection $pP \cap qP$ is either empty or it can be written as $\bigcup_{i=1}^n s_i P$ for some $n \in \bN$ and $s_i \in P$.
Hence right LCM semigroups are those that are at most \emph{singly aligned}.
A product system $X$ over such a semigroup is called \emph{compactly aligned} if it satisfies an analogous property to that of Fowler \cite{Fow02}, see Definition~\ref{D:famin}.
This is trivially satisfied for example when all left actions are by compacts; hence this class covers the semigroup algebras.
In terms of exactness we prove the following theorem.

\begin{thmintro}[Theorem~\ref{T:faexact}]
Let $P$ be a unital finitely aligned subsemigroup of an exact discrete group $G$ and let $X$ be a compactly aligned saturated product system over $P$ with coefficient C*-algebra $A$.
Then $\T_\la(X)$ is exact if and only if $A$ is exact.

Suppose in addition that $G$ is amenable, and $t$ is an injective Fock covariant representation of $X$.
Then $\ca(t)$ is exact if and only if $A$ is exact.
\end{thmintro}

Concerning nuclearity we look at \emph{strongly finitely aligned} semigroups, i.e., finitely aligned semigroups where the generation is achieved by mutually disjoint principal ideals.
There are examples that can be constructed using embeddings of the free semigroup, see Example~\ref{E:BS}.
In this setting compact alignment allows us to have an embedding of a $\bo{K}$-core inside the direct sum of images of relevant compact operators.
With that in hand we prove the following result.

\begin{thmintro}[Theorem~\ref{T:fanuc}]
Let $P$ be a unital strongly finitely aligned subsemigroup of an amenable discrete group $G$ and let $X$ be a compactly aligned saturated product system over $P$ with coefficient C*-algebra $A$.
Let $t$ be an injective equivariant Fock covariant representation of $X$.
Then $\ca(t)$ is nuclear if and only if the embedding $A \hookrightarrow [\ca(t)]_{e_P}$ is nuclear.
\end{thmintro}

These results apply in particular to the strong covariant representations of a product system.
Strong covariance was identified by Sehnem \cite{Seh19} as the required structure to deduce injectivity of representations from the coefficient algebra to the fixed point algebra.
The (reduced) strong covariant C*-algebra contains all previously known Cuntz-type C*-algebras that play the role of a boundary quotient.
This reduced model was further used in \cite{DKKLL22, KKLL23} to address the coaction C*-envelope problem, and by Sehnem \cite{Seh22} to resolve the C*-envelope problem in full generality.

\subsection{The compact alignment assumption}

One of the main goals in this endeavour has been to deduce exactness and nuclearity of $\ca(t)$ from the coefficient algebra $A$, at least when $t$ is Fock covariant \cite{Kak20, KKLL23, KLS25, Kat04}.
Compact alignment may look like a forced fix to obtain a positive answer to item (ii) of Question~\ref{Qn:main2}, but this is not simply out of convenience.
In Example~\ref{E:nonexact} we present a product system over $P = \bZ_+^2$ and $A = \bC$ that is not compactly aligned, \emph{and} its Fock C*-algebra fails to be exact (let alone be nuclear) as the core on $(1,1) + \bZ_+^2$ contains a copy of the non-exact free group C*-algebra $\ca(\bF_2)$.
Therefore we cannot hope for a characterisation in the spirit of Theorem D and Theorem E in the absence of compact alignment.

\subsection{Contents}

In Section~\ref{S:nuc} we give the background on nuclear maps that we require.
Here we show a gluing theorem for exactness that is iteratively used for summing up exact $\bo{K}$-cores.
This follows from the connection of exactness with locally split extensions.
We also present an analogous gluing theorem for nuclear embeddings, which is used in a similar manner.
See the AI Statement (Subsection~\ref{Ss:aist}) below for this result.

In Section~\ref{S:ps} we pass to the product systems results with a focus on semigroups that are embeddable in an amenable group.
We give the preliminaries, and we show that exactness (resp.\ nuclearity) of the fixed point algebra is equivalent to exactness (resp.\ nuclear embedding) of the $\bo{K}$-cores.
This resolves item (i) of Question~\ref{Qn:main2}.

In Section~\ref{S:capsfa} we study item (ii) of Question~\ref{Qn:main2}. 
In particular, we show that exactness (resp.\ nuclear embeddability in the fixed point algebra) of the $\bo{K}$-cores in the setting of compactly aligned product systems over finitely aligned (resp.\ strongly finitely aligned) semigroups is equivalent to exactness of $A$ (resp.\ nuclear embeddability of $A$ in the fixed point algebra).

In Section~\ref{S:beyond} we present an example where exactness of $\T_\la(X)$ fails to be inherited from $A$ even for $\bZ_+^2$ if the compact alignment assumption is dropped.
See the AI Statement (Subsection~\ref{Ss:aist}) below for this example.

\subsection{AI Statement}\label{Ss:aist}

The initial aim of the project has been to resolve in full generality the connection of exactness and nuclearity of the generated C*-algebras with the cores, and then with the coefficient C*-algebra.
The strategy was laid out in May 2023, in the second author's funding application to the Hellenic Foundation for Research and Innovation (HFRI) under the 5th Call for HFRI PhD Fellowships.
The exactness results of Subsection~\ref{Ss:exkcores} were achieved in December 2025.
The strategy to achieve the exactness (resp.\ nuclearity) results for compactly aligned product systems over finitely aligned semigroups (resp.\ over strongly finitely aligned semigroups) was set up at this point.

In August 2026, upon revisiting the nuclearity part of the problem, we prompted Large Language Models (LLMs) concerning a two-out-of-three nuclear-embedding lemma analogous to \cite[Proposition A.6]{Kat04} and \cite[Proposition 3.3]{Kak20} that have been used for this type of results in \cite{Kak20, KKLL23, Kat04}. 
After several unsuccessful attempts we prompted LLMs towards gluing nuclear embeddings, and after several iterations, we arrived at Theorem~\ref{T:Nucsum}.
Although Theorem~\ref{T:Nucsum} supersedes Theorem~\ref{T:Exsum}, we decided to leave the exposition as is.
Moreover, upon completing the project we prompted LLMs to provide examples of strongly finitely aligned semigroups that are not right LCM semigroups, and to test the exactness and nuclearity questions beyond compact alignment.
This returned Example~\ref{E:BS} and Example~\ref{E:nonexact}, respectively.

For the aforementioned results (Theorem~\ref{T:Nucsum}, Example~\ref{E:BS} and Example~\ref{E:nonexact}) and for editing and proofreading purposes, we used OpenAI's GPT-5.6 Sol and Anthropic's Claude Opus 5.5.
We reviewed the arguments as needed and rewrote the exposition to fit the context of the manuscript. 
We take full responsibility for the contents of the paper.

\section{Exactness and nuclear embeddings} \label{S:nuc}

\subsection{Notation}

Throughout, all C*-algebras are complex, ideals are closed and two-sided.
For two C*-algebras $\A, \B$ we denote by $\A\odot \B$ their algebraic tensor product, by $\A \otimes \B$ their minimal (spatial) tensor product, and by $\A\otimes_{\max} \B$ their maximal tensor product.
Recall that every C*-norm on $\A \odot \B$ is dominated by $\| \cdot \|_{\max}$ and dominates $\| \cdot \|_{\min}$, and that the identity map of $\A \odot \B$ therefore induces a $*$-epimorphism $\A \otimes_{\max} \B \to \A \otimes \B$. 
See \cite{BO08} for more details.

We will use the notion of the \emph{unconditional unitisation} $\A^\#$ of a C*-algebra $\A$.
For an embedding $\A \subseteq \B(H)$ we write $\A^\# = \A + \bC I_{H \oplus \bC} \subseteq \B(H \oplus \bC)$ with the understanding that we use the non-unital embedding of $\B(H)$ as a corner of $\B(H \oplus \bC)$.
If $\A$ is non-unital, then $\A^\#$ is also $*$-isomorphic to $\A + \bC I_{H}$.
If $\A$ is unital, then $\A^\#$ is $*$-isomorphic to the direct sum $\A \oplus \bC \subseteq \B(H \oplus \bC)$.

\subsection{Exactness} \label{Ss:exact}

We begin this subsection with some preliminaries on exact C*-algebras.
The next definitions and results regarding exactness of C*-algebras can be found in \cite{BO08}. 

\begin{definition}
Let $\A$ be a C*-algebra. 
We say that $\A$ is \emph{exact} if for any C*-algebra $\B$ and any ideal $\I \lhd \B$ the sequence
\[
\begin{tikzcd}
0 \arrow[r] & \A \otimes \I \arrow[r] & \A \otimes \B \arrow[r] & \A \otimes (\B/\I) \arrow[r] & 0
\end{tikzcd}
\]
is short exact.
\end{definition}

A deep result of Kirchberg \cite{Kir91, Kir94, Kir95} characterises exactness in terms of finite-dimensional approximations. 
For a detailed proof see \cite[Section 3.9]{BO08}.

\begin{theorem}\cite{Kir91, Kir94, Kir95}
Let $\A$ be a C*-algebra. 
Then $\A$ is exact if and only if there exist a faithful $*$-representation $\pi \colon \A \to \B(H)$ and nets of completely contractive completely positive maps $\phi_\la \colon \A \to M_{k_\la}(\bC)$ and $\psi_\la \colon M_{k_\la}(\bC) \to \B(H)$ such that 
\[
\lim_\la \| \psi_\la \circ \phi_\la (a) - \pi(a) \| = 0 \foral a\in \A.
\]
\end{theorem}

We will frequently use several properties of exactness. 
For detailed proofs see \cite{BO08}.
For example, a C*-algebra is exact if and only if its unconditional unitisation is exact. 
Furthermore, exactness passes to C*-subalgebras and quotients, and the class of exact C*-algebras is closed under taking finite direct sums, inductive limits and split extensions.

We will need a stronger result regarding extensions.
A short exact sequence 
\[
\begin{tikzcd}
0 \arrow[r] & \J \arrow[r] & \A \arrow[r, "q"] & \A/\J \arrow[r] & 0
\end{tikzcd}
\]
is called \emph{locally split} if for every finite-dimensional operator space $E \subseteq \A/\J$ there exists a completely contractive map $\theta \colon E \to \A$ such that $q\circ \theta = \id_E$. 
We will need the following non-unital version of \cite[Proposition 3.7.6]{BO08}.
The proof follows in the same way by replacing ``finite-dimensional operator systems'' by ``finite-dimensional operator spaces'', and ``completely contractive completely positive maps'' by ``completely contractive maps''.

\begin{proposition}\cite[Proposition 3.7.6]{BO08}\label{P:localsplit} 
Let $\A$ be a C*-algebra and for an ideal $\J \lhd \A$ let
\begin{equation*}
\begin{tikzcd}
0 \arrow[r] & \J \arrow[r] & \A \arrow[r, "q"] & \A/\J \arrow[r] & 0 
\end{tikzcd}
\end{equation*}
be a locally split short exact sequence. 
Then for any C*-algebra $\B$ we have that the sequence
\begin{equation*}
\begin{tikzcd}
0 \arrow[r] & \J\otimes \B \arrow[r] & \A\otimes \B \arrow[r, "q\otimes \id_{\B}"] & (\A/\J)\otimes \B \arrow[r] & 0 
\end{tikzcd}
\end{equation*}
is short exact.
\end{proposition}

The following two theorems give the characterisation of exactness through local reflexivity and local splits, see for example \cite[Theorem 10.2.4]{BO08}.
We include a proof for the unital and the non-unital versions separately.

\begin{theorem}\label{T:Exlocunital}\cite[unital version of Theorem 10.2.4]{BO08}
Let $\A$ be a unital C*-algebra and for an ideal $\J \lhd \A$ let
\begin{equation*}
\begin{tikzcd}
0 \arrow[r] & \J \arrow[r] & \A \arrow[r, "q"] & \A/\J \arrow[r] & 0 
\end{tikzcd}
\end{equation*}
be a short exact sequence with $\J$ and $\A/\J$ exact. 
The following are equivalent:
\begin{enumerate}
\item The short exact sequence is locally split.
\item For every finite-dimensional unital operator system $E \subseteq \A/\J$ there exists a unital completely positive map $\theta \colon E \to \A$ such that $q \circ \theta = \id_E$.
\item For every finite-dimensional unital operator system $E \subseteq \A/\J$ there exists a completely contractive completely positive map $\theta \colon E \to \A$ such that $q \circ \theta = \id_E$.
\item $\A$ is exact.
\end{enumerate}
\end{theorem}

\begin{proof}
The implication [(i) $\Rightarrow$ (iv)] is \cite[Exercise 3.9.8]{BO08}.
We provide a proof to show how Proposition~\ref{P:localsplit} is used.
Suppose that $\A/\J$ and $\J$ are exact and that the short exact sequence is locally split. 
Let $\I$ be an ideal in a C*-algebra $\B$ and denote by $q'\colon \B \to \B/\I$ the quotient map. 
Consider the following commutative diagram
\[
\begin{tikzcd}
& 0 \arrow[d] 
& 0 \arrow[d] 
& 0 \arrow[d] 
& \\
0 \arrow[r] 
& \J\otimes \I \arrow[r] \arrow[d] 
& \J\otimes \B \arrow[r] \arrow[d] 
& \J\otimes (\B/\I) \arrow[r] \arrow[d] 
& 0 \\
0 \arrow[r] 
& \A\otimes \I \arrow[r] \arrow[d] 
& \A\otimes \B \arrow[r] \arrow[d] 
& \A\otimes (\B/\I) \arrow[d] \\
0 \arrow[r] 
& (\A/\J)\otimes \I \arrow[r] \arrow[d] 
& (\A/\J)\otimes \B \arrow[r] \arrow[d] 
& (\A/\J)\otimes (\B/\I) \arrow[r] \arrow[d] 
& 0 \\
& 0
& 0
& 0
&
\end{tikzcd}
\]
where exactness of the top and bottom rows follows from exactness of the C*-algebras $\J$ and $\A/\J$, respectively, and exactness of the columns follows from Proposition~\ref{P:localsplit}. 
We have to show that the middle row is also exact. 

To that end, let $x$ be in $\ker(\id_\A\otimes q')$ and we will show that $x\in \A\otimes \I$. 
By commutativity of the diagram and exactness of the bottom row we obtain 
\[
(q\otimes \id_\B)(x) \in (\A/\J)\otimes \I,
\]
hence we may pick $y\in \A\otimes \I$ such that $(q\otimes \id_\I) (y)= (q\otimes \id_\B)(x)$. 
Set $z:= x-y$ and note that $(q\otimes \id_\B) (z)=0$. 
In particular, by exactness of the middle column we obtain that $z\in \J\otimes \B$. 
Since $z\in \ker(\id_\A \otimes q')$, exactness of the top row implies that $z\in \J \otimes \I \subseteq \A \otimes \I$.
We conclude that 
\[
x= x-y+y=z+y \in \A \otimes \I,
\]
as required.

The equivalence [(iii) $\Leftrightarrow$ (iv)] is the version proven in \cite[Theorem 10.2.4]{BO08} if one follows verbatim the proofs for finite-dimensional operator subsystems and completely contractive completely positive maps.

The implication [(iii) $\Rightarrow$ (ii)] is a standard trick exhibited in the comments prior to \cite[Proposition C.2]{BO08}.
Namely, if for a finite-dimensional unital operator system $E \subseteq \A/\J$ there is a completely contractive completely positive map $\theta\colon E \to \A$ such that $q \circ \theta = \id_E$, then let $\om \colon E \to \bC$ be a unital completely positive map and replace $\theta$ by $\theta + (1_\A -\theta(1_{\A/\J}))\om$.

For the implication [(ii) $\Rightarrow$ (i)] let $E \subseteq \A/\J$ be a finite-dimensional operator space and apply the hypothesis for the unital operator system $E'$ generated by $E$ to obtain a unital completely positive map $\theta'\colon E' \to \A$ such that $q \circ \theta' = \id_{E'}$.
Then the completely contractive restriction map $\theta := \theta'|_E$ is the required local lift.
\end{proof}

\begin{theorem}\label{T:Exlocgen}\cite[Theorem 10.2.4]{BO08}
Let $\A$ be a C*-algebra and for an ideal $\J \lhd \A$ let
\begin{equation*}
\begin{tikzcd}
0 \arrow[r] & \J \arrow[r] & \A \arrow[r, "q"] & \A/\J \arrow[r] & 0 
\end{tikzcd}
\end{equation*}
be a short exact sequence with $\J$ and $\A/\J$ exact. 
The following are equivalent:
\begin{enumerate}
\item The short exact sequence is locally split.
\item For every finite-dimensional selfadjoint subspace $E \subseteq \A/\J$ there exists a completely contractive completely positive map $\theta \colon E \to \A$ such that $q \circ \theta = \id_E$.
\item $\A$ is exact.
\end{enumerate}
\end{theorem}

\begin{proof}
The implication [(i) $\Rightarrow$ (iii)] is the same as [(i) $\Rightarrow$ (iv)] of Theorem~\ref{T:Exlocunital}.

For the implication [(iii) $\Rightarrow$ (ii)] we write $\A^\#$ for the unconditional unitisation of $\A$.
Since $\A$ is exact, $\A^\#$ is exact and therefore so is $\A^\#/\J$.
Consider the short exact sequence 
\[
\begin{tikzcd}
0 \arrow[r] & \J \arrow[r] & \A^\# \arrow[r, "q^\#"] & \A^\#/\J \arrow[r] & 0
\end{tikzcd}
\]
where $q^\#$ is the natural unital map that extends $q$. 

Let $E\subseteq \A/\J$ be a finite-dimensional selfadjoint subspace, and denote by $\wt{E}$ the operator system generated by $E$ inside $\A^\#/\J$. 
Since $\A^\#$ is exact, by Theorem~\ref{T:Exlocunital} we obtain a unital completely positive map $\wt{\theta}\colon \wt{E} \to \A^\#$ such that $q^\#\circ\wt{\theta} = \id_{\wt{E}}$. 
Consider the commutative diagram
\[
\begin{tikzcd}
\A \arrow[r] \arrow[d, "q"] & \A^\# \arrow[d,"q^\#"] \\
\A/\J \arrow[r] & \A^\#/\J
\end{tikzcd}
\]
For $x\in E$ we have $\wt{\theta}(x) = a + \mu 1_{\A^\#}$ where $a\in \A$ and $\mu \in \bC$. 
Moreover, we have
\[
x = q^\#\circ \wt{\theta}(x) = q(a)+\mu 1_{\A^\#/\J},
\]
which implies that $\mu = 0$ as the unit of $\A^\#/\J$ is not contained in $\A/\J$ (being the extension by one dimension as a vector space). 
In particular, we get $\wt{\theta}(E)\subseteq \A$. 
Set 
\[
\theta:= \wt{\theta}|_E \colon E \to \A,
\]
and note that $\theta$ is the required completely contractive and completely positive map satisfying 
\[
q \circ \theta (x) = q^\#|_\A \circ \wt{\theta}|_E(x) = x \foral x\in E.
\]

For the implication [(ii) $\Rightarrow$ (i)] let $E \subseteq \A/\J$ be a finite-dimensional operator space and apply the hypothesis for the selfadjoint subspace $E'$ generated by $E$ to obtain a completely contractive completely positive map $\theta'\colon E' \to \A$ such that $q \circ \theta' = \id_{E'}$.
Then the completely contractive restriction map $\theta := \theta'|_E$ is the required local lift.
\end{proof}

We now show the following gluing theorem for exactness.

\begin{theorem}\label{T:Exsum}
Let $\C$ be a C*-algebra and let $\J$ and $\A$ be C*-subalgebras of $\C$ such that $\J + \A$ is a C*-subalgebra of $\C$ and $\J \lhd \J + \A$.
If $\J$ and $\A$ are exact, then so is $\J + \A$.
\end{theorem}

\begin{proof}
Let $\io \colon \A \to \J + \A$ be the canonical embedding and $\pi \colon (\J+\A)/\J \to \A/(\J\cap \A)$ be the canonical $*$-isomorphism. 
We obtain the following commutative diagram
\[
\begin{tikzcd}
0 \arrow[r] & \J\cap \A \arrow[r] & \A \arrow[r, "q_{\A}"] \arrow[d, "\io"] & \A/(\J \cap \A) \arrow[r] & 0 \\
0 \arrow[r] & \J \arrow[r] & \J + \A \arrow[r, "q"] & (\J + \A)/\J \arrow[r] \arrow[u, "\pi"] & 0
\end{tikzcd}
\]
and since $\A$ is exact, we get that $\J \cap \A$ and $\A/(\J \cap \A) \simeq (\J+\A)/\J$ are also exact. By Theorem~\ref{T:Exlocgen} we thus obtain that the top short exact sequence is locally split.
By assumption $\J$ is also exact, and hence by Theorem~\ref{T:Exlocgen} it suffices to show that the bottom short exact sequence is locally split.

Towards this end, let $E \subseteq (\J+\A)/\J$ be a finite-dimensional operator space.
Let $\theta \colon \pi(E) \to \A$ be a completely contractive map such that $q_\A \circ \theta (\pi(x)) = \pi(x)$ for all $x \in E$.
We then have
\[
q \circ \io \circ \theta \circ \pi(x) = \pi^{-1} \circ q_\A \circ \theta \circ \pi(x) = x \foral x \in E.
\]
Since $\io\circ\theta\circ\pi|_E$ is also completely contractive the proof is complete. 
\end{proof}

We close with the following remark that will be used subsequently.

\begin{remark} \label{R:Ex-ext}
Let $\A$ be an exact C*-algebra and $\J \subseteq \A$ be an ideal.
Let $E \subseteq \A/\J$ be a finite-dimensional selfadjoint subspace and let $\theta \colon E \to \A$ be the completely contractive completely positive local lift of $\id_E$ obtained in the proof of Theorem~\ref{T:Exlocgen}.
From the proof of Theorem~\ref{T:Exlocgen} the lift $\theta$ is given as the restriction of a unital completely positive map $\wt{\theta}\colon E + \bC 1_{\A^\#/\J} \to \A^\#$ for the unconditional unitisation $\A^\#$ of $\A$.
Suppose we are given a completely contractive and completely positive map 
\[
\al \colon \A \to M_n(\bC).
\]
We want to show that the completely contractive completely positive map 
\[
\al \circ \theta \colon E \to M_n(\bC)
\]
has a completely contractive completely positive extension to the entire $\A/\J$.

Towards this end, we may consider the canonical extension $\al^\# \colon \A^\# \to M_n(\bC)$ which is a unital completely positive map, see for example \cite[Lemma 4.9 (c) and Corollary 4.17]{Wer02}.
Then we obtain a unital completely positive map
\[
\al^\# \circ \wt{\theta} \colon E + \bC 1_{\A^\#/\J} \to M_n(\bC),
\]
which by Arveson's Extension Theorem has a unital completely positive extension
\[
\wt{\al^\# \circ \wt{\theta}} \colon \A^\#/\J \to M_n(\bC).
\]
By construction we have that $\wt{\al^\# \circ \wt{\theta}}|_{\A/\J}$ is a completely contractive completely positive extension of $\al \circ \theta$. 
\end{remark}

\subsection{Nuclear embeddings}

We next pass to the theory of nuclearity and nuclear maps.

\begin{definition}
Let $\A$ be a C*-algebra. 
We say that $\A$ is \emph{nuclear} if for every C*-algebra $\B$ the canonical $*$-epimorphism $\A\otimes_{\max} \B \to \A\otimes \B$ is a $*$-isomorphism.
\end{definition}

The following remarkable result was obtained independently by Kirchberg \cite{Kir77} and by Choi--Effros \cite{CE78}.

\begin{theorem}\cite{CE78, Kir77}
Let $\A$ be a C*-algebra.
Then $\A$ is nuclear if and only if there exist nets of completely contractive completely positive maps $\phi_\la \colon \A \to M_{k_\la}(\bC)$ and $\psi_\la \colon M_{k_\la}(\bC) \to \A$ such that 
\[
\lim_\la \| \psi_\la \circ \phi_\la (a) - a \| = 0 \foral a\in \A.
\]
\end{theorem}

We now give the definition of nuclear maps and record some of their standard properties. For more details see \cite{BO08}.

\begin{definition}
Let $\C$ and $\D$ be C*-algebras.
A completely contractive completely positive map $\Phi \colon \C \to \D$ is called \emph{nuclear} if there exist nets of completely contractive completely positive maps $\phi_\la \colon \C \to M_{k_\la}(\bC)$ and $\psi_\la \colon M_{k_\la}(\bC) \to \D$ such that $\lim_\la \| \psi_\la \circ \phi_\la (c) - \Phi(c) \| = 0 \foral c \in \C$.

If $\C$ is a C*-subalgebra of $\D$ and the inclusion map $\C \hookrightarrow \D$ is nuclear, then we say that \emph{$\C$ is nuclearly embedded in $\D$}.
\end{definition}

Equivalently, a map $\Phi \colon \C \to \D$ between C*-algebras is nuclear if and only if for every finite set $\F\subseteq \C$ and $\eps>0$ we may find $n\in \bN$ and completely contractive completely positive maps $\phi \colon \C \to M_n(\bC)$ and $\psi\colon M_n(\bC) \to \D$ such that 
\[
\| \psi \circ \phi (x) - \Phi(x) \|< \eps \foral x \in \F.
\]
Note that the matrix algebras appearing in the definition of a nuclear map can be substituted by any finite-dimensional C*-algebra.
Moreover, a C*-algebra $\C$ is nuclear (resp.\ exact) if and only if $\id_\C$ is a nuclear map (resp.\ it has a nuclear faithful $*$-representation); hence every finite-dimensional C*-algebra is nuclear (and thus also exact).
The composition of a nuclear map with a completely contractive completely positive map is also nuclear; hence (injective) $*$-homomorphisms from nuclear C*-algebras are nuclear.

Concerning direct sums, if $\C_i \subseteq \D$ is a nuclear embedding for $i=1,\dots,n$ and the $\C_i$'s are mutually orthogonal, then $\C:=\bigoplus_{i=1}^n \C_i$ is nuclearly embedded in $\D$.
If $\D$ is the C*-inductive limit of an increasing net of C*-subalgebras $(\D_i)_{i \in I}$, then $\D$ is nuclear if and only if $\D_i$ is nuclearly embedded in $\D$ for every $i \in I$.
For detailed proofs see \cite{BO08}.

We next see a gluing theorem for nuclear embeddings.

\begin{theorem}\label{T:Nucsum}
Let $\C$ be a C*-algebra and let $\J$ and $\A$ be C*-subalgebras of $\C$ such that $\J + \A$ is a C*-subalgebra of $\C$ and $\J \lhd \J + \A$.
If $\J$ and $\A$ are nuclearly embedded in $\C$, then so is $\J + \A$.
\end{theorem}

\begin{proof}
To fix notation let
\[
q \colon \J + \A \to (\J+\A)/\J
\qand
q_\A \colon \A \to \A/(\J \cap \A)
\]
be the canonical quotient maps, and consider the $*$-isomorphism
\[
\pi \colon (\J+\A)/\J \to \A/(\J\cap\A),
\]
that satisfies $\pi \circ q|_\A = q_\A$.
Henceforth, fix a finite set $\F \subseteq \J+\A$ of contractions and $\eps>0$, and consider the finite-dimensional selfadjoint subspace
\[
E := \spn\{ \pi \circ q (x), \pi \circ q(x)^* \mid x \in \F \} \subseteq \A/(\J\cap\A).
\]
Since $\A$ is nuclearly embedded in $\C$ it is also exact, and hence so are $\J \cap \A$ and $\A/(\J \cap \A)$.
Therefore Theorem~\ref{T:Exlocgen} applies to the short exact sequence
\begin{equation*}
\begin{tikzcd}
0 \arrow[r] & \J \cap \A \arrow[r] & \A \arrow[r,"q_{\A}"] & \A/(\J \cap \A) \arrow[r] & 0
\end{tikzcd}
\end{equation*}
and yields a completely contractive completely positive map $\theta \colon E \to \A$ such that $q_\A \circ \theta = \id_E$.
Since $\theta(\pi \circ q(x)) \in \A$ we get
\[
\pi \circ q \circ \theta(\pi \circ q(x)) = q_\A \circ \theta(\pi \circ q(x)) = \pi \circ q(x),
\]
and since $\pi$ is a $*$-isomorphism we get
\[
q \circ \theta(\pi \circ q(x)) = q(x).
\]
Hence we obtain
\[
b_x := \theta(\pi \circ q(x)) - x \in \J \foral x \in \F.
\]

Since $\A$ is nuclearly embedded in $\C$, for the finite set of contractions $\{ \theta(\pi \circ q(x)) \mid x \in \F\}$ and for the fixed $\eps>0$ there exist $n \in \bN$ and completely contractive completely positive maps $\al_\A \colon \A \to M_n(\bC)$ and $\be_\A \colon M_n(\bC) \to \C$ such that
\begin{equation}\label{Eq:nucquot}
\| \be_\A \circ \al_\A (\theta(\pi \circ q(x))) - \theta(\pi \circ q(x)) \| < \eps \foral x \in \F.
\end{equation}
By Remark~\ref{R:Ex-ext} there exists a completely contractive completely positive map 
\[
\wt{\al_\A \circ \theta} \colon \A/(\J \cap \A) \to M_n(\bC) \text{ such that } \wt{\al_\A \circ \theta}|_E = \al_\A \circ \theta.
\]
By setting
\[
\wh{\al}_\A := \wt{\al_\A \circ \theta} \circ \pi \circ q \colon \J + \A \to M_n(\bC),
\]
we obtain a completely contractive completely positive map with
\[
\wh{\al}_\A(x) = \wt{\al_\A \circ \theta}(\pi \circ q (x)) = \al_\A(\theta(\pi \circ q(x)))
\foral x \in \F.
\]
Using \eqref{Eq:nucquot} we obtain
\begin{equation}\label{Eq:nucGa}
\| \be_\A \circ \wh{\al}_\A (x) - x - b_x \| = \| \be_\A \circ \al_\A(\theta(\pi \circ q(x))) - \theta(\pi \circ q(x)) \| < \eps \foral x \in \F.
\end{equation}

We next use the canonical embeddings
\[
\J^\# \subseteq (\J + \A)^\# \subseteq \C^\#
\]
for the unconditional unitisations.
Fix continuous functions $u, v \colon [0,1] \to [0,1]$ such that
\[
u(s) = 0 \qfor s \leq 1/2,
\quad
u(1) = 1,
\quad
v(s) = 0 \qfor s \leq 1/4,
\quad
v(s) = 1 \qfor s \geq 1/2,
\]
and set $k := (1 - u^2)^{1/2}$.
By construction we have
\begin{equation}\label{Eq:nucfun}
u^2 + k^2 = 1,
\quad
u v = u,
\quad
u(0) = v(0) = 0
\qand
k(1) = 0.
\end{equation}
By \cite[Theorem 3.12.14]{Ped79} we may pick an approximate unit $(e_\la)_\la$ of $\J$ consisting of positive contractions which is quasicentral in $\J+\A$, that is
\[
\lim_\la \| e_\la y - y e_\la \| = 0 \foral y \in \J + \A.
\]

We will use two standard consequences of the continuous functional calculus, which we perform in $(\J+\A)^\#$.
Note that since each $e_\la$ is a positive contraction we have $\si(e_\la) \subseteq [0,1]$, and therefore
\begin{equation}\label{Eq:nucfc}
\| f(e_\la) \| \leq \|f\|_\infty := \sup\{ |f(s)| \mid s \in [0,1] \} \foral f \in C[0,1].
\end{equation}
 
 Firstly, quasicentrality passes to the continuous functional calculus, in the sense that
\begin{equation}\label{Eq:nucqc}
\lim_\la \| f(e_\la) y - y f(e_\la) \| = 0 \foral y \in \J+\A \AND f \in C[0,1].
\end{equation}
Indeed, for $j \geq 1$ the telescoping identity
\[
e_\la^j y - y e_\la^j = \sum_{i=0}^{j-1} e_\la^i \, ( e_\la y - y e_\la ) \, e_\la^{j-1-i}
\]
together with $\|e_\la\| \leq 1$ gives 
\[
\|e_\la^j y - y e_\la^j\| \leq j \|e_\la y - y e_\la\|.
\]
As scalar multiples of the unit are central, for every polynomial $p(s) = \sum_{j=0}^N c_j s^j$ we obtain
\[
\| p(e_\la) y - y p(e_\la) \| \leq \Big( \sum_{j=1}^N j |c_j| \Big) \| e_\la y - y e_\la \| \xrightarrow{\ \la \ } 0.
\]
Now let $f \in C[0,1]$ and $\eps'>0$, and use the Stone--Weierstrass Theorem to pick a polynomial $p$ with $\|f - p\|_\infty < \eps'$, so that $\|f(e_\la) - p(e_\la)\| < \eps'$ for all $\la$ by \eqref{Eq:nucfc}.
Then
\[
\limsup_\la \| f(e_\la) y - y f(e_\la) \|
\leq
2 \eps' \|y\| + \limsup_\la \| p(e_\la) y - y p(e_\la) \|
=
2 \eps' \|y\|,
\]
and letting $\eps' \to 0$ establishes \eqref{Eq:nucqc}.
 
Secondly, since $(e_\la)_\la$ is an approximate unit for $\J$ we have
\begin{equation}\label{Eq:nucau}
\lim_\la \| f(e_\la) d - f(1) d \| = 0 \foral d \in \J \AND f \in C[0,1].
\end{equation}
Indeed, for $f(s) = s^j$ with $j \geq 1$ this follows by induction on $j$ from the identity
\[
e_\la^j d - d = e_\la^{j-1} ( e_\la d - d ) + ( e_\la^{j-1} d - d )
\]
and $\|e_\la\| \leq 1$, and it also holds trivially for $j=0$.
By linearity it holds for $f$ being a polynomial.
For arbitrary $f \in C[0,1]$ and $\eps' > 0$ pick a polynomial $p$ with $\|f - p\|_\infty < \eps'$.
Then we get $|f(1) - p(1)| < \eps'$, and we also have
\[
\| f(e_\la) d - f(1) d \|
\leq
\| f(e_\la) - p(e_\la) \| \, \|d\| + \| p(e_\la) d - p(1) d \| + | p(1) - f(1) | \, \|d\|.
\]
Therefore by using \eqref{Eq:nucfc} we get
\[
\limsup_\la \| f(e_\la) d - f(1) d \| \leq 2 \eps' \|d\|,
\]
and \eqref{Eq:nucau} follows.
Applying \eqref{Eq:nucau} to $f = k$ and using $k(1) = 0$ from \eqref{Eq:nucfun} we obtain
\begin{equation}\label{Eq:nuck}
\lim_\la \| k(e_\la) d \| = 0 \foral d \in \J.
\end{equation}
 
Using \eqref{Eq:nucqc} and \eqref{Eq:nuck} we may fix $\la$ such that the positive contractions
\[
U := u(e_\la) \in \J,
\quad
V := v(e_\la) \in \J
\qand
K := k(e_\la) \in (\J+\A)^\#
\]
satisfy
\begin{equation}\label{Eq:nucpart}
\| x - U x U - K x K \| < \eps
\qand
\| K b_x K \| < \eps
\foral x \in \F.
\end{equation}
Indeed, the fact that $U, V \in \J$ follows by noting that $u(0)=v(0)=0$ (and hence $u$, $v$ can be taken as limits of polynomials with zero constant term) and $e_\la \in \J$. 
The first estimate follows by using the equality
\[
U x U + K x K = U^2 x + K^2 x + U[x,U] + K[x,K] = x + U[x,U]+K[x,K],
\]
which is obtained by \eqref{Eq:nucfun}, and noting that both commutators vanish in the limit by \eqref{Eq:nucqc}. 
The second estimate follows from \eqref{Eq:nuck} and the fact that $b_x\in \J$.

Since $\J$ is an ideal of $\J + \A$ we have $V y V \in \J$ for every $y \in \J + \A$. 
In particular, we have that the finite set $\{V x V \mid x\in \F\}$ sits inside $\J$.
Using that $\J$ is nuclearly embedded in $\C$, we may pick $m \in \bN$ and completely contractive completely positive maps $\al_\J \colon \J \to M_m(\bC)$ and $\be_\J \colon M_m(\bC) \to \C$ such that
\begin{equation}\label{Eq:nucid}
\| \be_\J \circ \al_\J (V x V) - V x V \| < \eps \foral x \in \F.
\end{equation}

To complete the proof, consider the maps
\begin{align*}
\Phi & \colon \J + \A \to M_m(\bC) \oplus M_n(\bC);\
\Phi(y) = \al_\J(V y V) \oplus \wh{\al}_\A(y),
\end{align*}
and
\begin{align*}
\Psi & \colon M_m(\bC) \oplus M_n(\bC) \to \C;\
\Psi(y_1 \oplus y_2) = U \be_\J(y_1) U + K \be_\A(y_2) K.
\end{align*}
The map $\Psi$ indeed takes values in $\C$ as $\C$ is an ideal of $\C^\#$.
The map $\Phi$ is completely positive and completely contractive as a direct sum of such maps.
The map $\Psi$ is completely positive, and since $\be_\J$ and $\be_\A$ are contractive we have $\be_\J(1) \leq 1$ and $\be_\A(1) \leq 1$ in $\C^\#$, whence
\[
\Psi(1 \oplus 1) = U \be_\J(1) U + K \be_\A(1) K \leq U^2 + K^2 = 1.
\]
As completely positive maps on unital C*-algebras attain their cb-norm at the unit, we conclude that $\Psi$ is completely contractive.

It remains to estimate $\Psi \circ \Phi$ on $\F$.
Since $UV = VU = U$ by \eqref{Eq:nucfun}, equation \eqref{Eq:nucid} gives
\[
\| U \be_\J \circ \al_\J (V x V) U - U x U \|
=
\| U \big( \be_\J \circ \al_\J (VxV) - VxV \big) U \|
< \eps,
\]
while \eqref{Eq:nucGa} and the second estimate of \eqref{Eq:nucpart} give
\[
\| K \be_\A \circ \wh{\al}_\A(x) K - K x K \|
\leq
\| K \big( \be_\A \circ \wh{\al}_\A (x) - x - b_x \big) K \| + \| K b_x K \|
<
2 \eps.
\]
Combining with the first estimate of \eqref{Eq:nucpart}, for every $x\in \F$, we obtain
\begin{align*}
\| \Psi \circ \Phi(x) - x \|
& =
\| U \be_\J \circ \al_\J (V x V) U + K \be_\A \circ \wh{\al}_\A(x) K - x \| \\
& \leq
\| U \be_\J \circ \al_\J (V x V) U - U x U \|
+ \| K \be_\A \circ \wh{\al}_\A(x) K - K x K \| + \\
& \hspace{5.34cm} + \| U x U + K x K - x \| \\
& <
\eps + 2 \eps + \eps
=
4 \eps.
\end{align*}
Since $\F$ and $\eps$ were arbitrary, the inclusion of $\J + \A$ in $\C$ is nuclear and the proof is complete.
\end{proof}

\section{Product systems}\label{S:ps}

\subsection{Notation}

If $\X$ is a Banach space, then for a subset $\Y\subseteq \X$ we write $[ \Y ]$ for the closed linear span of $\Y$. 
Whenever $P$ is a unital subsemigroup of a discrete group $G$ we assume that $e_P=e_G$ for their units.

\subsection{Semigroups and constructible ideals} \label{Ss:sgp}

We include this short subsection to fix notation for semigroups and their constructible ideals. 
For more details see \cite{Li12}.
Let $P$ be a unital subsemigroup of a discrete group $G$.
We write
\[
P^* := \{ p \in P \mid p q = q p = e_P \text{ for some } q \in P\}
\]
for its group of units.
For $Z \subseteq P$ and $p \in P$ we write
\[
p Z := \{ p q \mid q \in Z \}
\qand
p^{-1} Z := \{ q \in P \mid p q \in Z \},
\]
so in particular $p^{-1} P = P$. 
The right ideals $pP$ for $p\in P$ are called \emph{principal} right ideals.
We write $\J$ for the smallest family of right ideals of $P$ containing $P$ and $\mt$ that is closed under left multiplication, under taking pre-images under left multiplication, and under finite intersections.
By \cite[Lemma 3.3]{Li12} this family admits the reduced form
\[
\J = \left\{ q_n^{-1} p_n \dots q_1^{-1} p_1 P \mid n \in \bN;
p_i, q_i \in P \right\} \cup \{\mt\},
\]
and its elements are called \emph{constructible} right ideals.
We write $\Bx, \By, \Bz$ etc.\ for the elements of $\J$.

We have $q^{-1} p Z = P \cap q^{-1} \cdot p \cdot Z$ for every $Z \subseteq P$, and inductively
\begin{equation} \label{Eq:faintersect}
q_n^{-1} p_n \dots q_1^{-1} p_1 P
=
P \cap (q_n^{-1} \cdot p_n \cdot P) \cap \cdots \cap
(q_n^{-1} \cdot p_n \cdots q_1^{-1} \cdot p_1 \cdot P).
\end{equation}
Finally, for every $\mt \neq \Bx \in \J$ one can choose $p_1, q_1, \dots, p_n, q_n \in P$ such that $p_1^{-1} q_1 \cdots p_n^{-1} q_n = e_P$ and $\Bx = q_n^{-1} p_n \dots q_1^{-1} p_1 P$.

\subsection{Product systems and the Fock representation}

Product systems are constructed over families of C*-correspondences.
We will be using the theory of C*-correspondences freely; the reader may refer to \cite{Lan95, MT05} for the standard theory of Hilbert C*-modules.

We use the framework of concrete product systems from \cite{KP25}, which simplifies notationally some of the arguments.
A family $X = \{X_p\}_{p \in P}$ of norm-closed operator spaces in a common $\B(H)$ is a \emph{(concrete) product system} if:
\begin{enumerate}
\item $A := X_{e_P}$ is a C*-algebra;
\item $X_p \cdot X_q \subseteq X_{pq}$ for all $p, q \in P$; and
\item $X_p^* \cdot X_{pq} \subseteq X_q$ for all $p, q \in P$.
\end{enumerate}
Each $X_p$ is then a C*-correspondence over $A$ with inner product $\sca{\xi_p, \eta_p} := \xi_p^* \cdot \eta_p$ and left action $\vphi_{p}$ given by multiplication.
For $\xi_p \in X_p$, the multiplication operator
\[
M_{\xi_p}^{q, pq} \colon X_q \to X_{pq};\ \eta_q \mapsto \xi_p \cdot \eta_q,
\]
lies in $\L(X_q, X_{pq})$, with adjoint $\eta_{pq} \mapsto \xi_p^* \cdot \eta_{pq}$.
We will say that $X$ is \emph{saturated} if $[X_p \cdot X_q] = X_{pq}$ for all $p,q \in P$.
In this case we have that the product system is automatically \emph{non-degenerate}, i.e., $[A \cdot X_p] = X_p$ for all $p \in P$.

On the Fock space $\F X := \sum^\oplus_{r \in P} X_r$, viewed as a right Hilbert $A$-module, we consider the \emph{left creation operators}
\[
\la_p(\xi_p) := \sum^\oplus_{r \in P} M_{\xi_p}^{r, pr} \colon \F X \to \F X,
\]
so that for $\eta_r \in X_r$ we have
\[
\la_p(\xi_p) \eta_r = \xi_p \cdot \eta_r
\qand
\la_p(\xi_p)^* \eta_r =
\begin{cases}
\xi_p^* \cdot \eta_r & \text{if } r \in pP, \\
0 & \text{if } r \notin pP.
\end{cases}
\]
We write $\la = \{\la_p\}_{p \in P}$ for the \emph{Fock representation of $X$} and $\T_\la(X) := \ca(\la_p(X_p) \mid p \in P)$ for the \emph{Fock C*-algebra}.

A \emph{(Toeplitz) representation} $t = \{t_p\}_{p \in P}$ of $X$ consists of linear maps $t_p \colon X_p \to \B(H')$ for every $p\in P$ such that:
\begin{enumerate}
\item $t_{e_P}$ is a $*$-representation of $A$;
\item $t_p(\xi_p) t_q(\eta_q) = t_{pq}(\xi_p \cdot \eta_q)$ for all $p, q \in P$ and $\xi_p \in X_p$, $\eta_q \in X_q$; and 
\item $t_p(\xi_p)^* t_{pq}(\eta_{pq}) = t_q(\xi_p^* \cdot \eta_{pq})$ for all $p, q \in P$ and $\xi_p \in X_p$, $\eta_{pq} \in X_{pq}$.
\end{enumerate}
A representation is called \emph{injective} if $t_{e_P}$ is injective. 
It follows by the definition of the Fock representation that it is an injective representation of $X$.
We write $\T(X)$ for the \emph{universal Toeplitz C*-algebra} with respect to the representations of $X$, and we write 
\[
t_\ast \colon \T(X) \to \ca(t);\ \wh{t}_p(\xi_p) \mapsto t_p(\xi_p) \foral p \in P,
\]
for the induced $*$-representation on $\T(X) = \ca(\hspace{2pt} \wh{t} \hspace{2pt})$. 
A representation $t$ is called \emph{equivariant} if there is a $*$-homomorphism 
\[
\de \colon \ca(t) \to \ca(t) \otimes \ca(G);\ \de(t_p(\xi_p)) = t_p(\xi_p) \otimes u_p \foral p \in P \text{ and } \xi_p \in X_p,
\]
where $\{u_g\}_{g\in G}$ are the unitaries that generate the group C*-algebra $\ca(G)$.
In this case $\ca(t)$ admits a coaction by $G$ whose $g$-fibre $[\ca(t)]_g$ is the closed linear span of the spaces
\[
\left( t_{p_1}(X_{p_1})^* \right)^{\eps} t_{q_1}(X_{q_1}) \cdots
t_{p_n}(X_{p_n})^* t_{q_n}(X_{q_n})^{\eps'}
\text{ with }
p_1^{-\eps} q_1 \cdots p_n^{-1} q_n^{\eps'} = g,
\]
for $\eps, \eps' \in \{0,1\}$ and $n \in \bN$. 
We call the $e_P$-fibre $[\ca(t)]_{e_P}$ the \emph{fixed point algebra} of $\ca(t)$. 
If, in addition, the map $(\id \otimes \la_G) \circ \de$ is injective, where $\la_G \colon \ca(G) \to \ca_\la(G)$ denotes the left regular representation of the group $G$, then $\de$ will be called \emph{normal}.
The universal representation of $X$ is equivariant and the Fock representation is equivariant with a normal coaction \cite[Proposition~4.1]{DKKLL22}.
Moreover, for an equivariant representation $t$ of $X$ the collection $\{[\ca(t)]_g\}_{g\in G}$ forms a Fell bundle over $G$ in the sense of \cite{Exe17}.

The following two results, which describe how words in the creation operators act on the fibres of $\F X$, are the basic computational tools.

\begin{proposition} \label{P:mult} \cite[Proposition 2.3]{KP25}
Let $P$ be a unital subsemigroup of a discrete group $G$ and let $X$ be a product system over $P$.
Let $t$ be a representation of $X$ and let $\xi_r \in X_r$, $\xi_{p_i} \in X_{p_i}$ and $\xi_{q_i} \in X_{q_i}$ for $r, p_1, q_1, \dots, p_n, q_n \in P$ and $\eps, \eps' \in \{0,1\}$.
If $r \in q_n^{-\eps'} p_n \dots q_1^{-1} p_1^{\eps} P$, then $p_1^{-\eps} q_1 \cdots p_n^{-1} q_n^{\eps'} r \in P$ and
\[
\left( t_{p_1}(\xi_{p_1})^* \right)^{\eps} t_{q_1}(\xi_{q_1}) \cdots
t_{p_n}(\xi_{p_n})^* t_{q_n}(\xi_{q_n})^{\eps'} t_r(\xi_r)
=
t_{p_1^{-\eps} q_1 \cdots p_n^{-1} q_n^{\eps'} r}
\left( (\xi_{p_1}^*)^{\eps} \xi_{q_1} \cdots \xi_{p_n}^*
\xi_{q_n}^{\eps'} \xi_r \right).
\]
\end{proposition}

\begin{proposition} \label{P:fcmt} \cite[page 340]{Fow02}, \cite[Proposition 2.4]{KP25}
Let $P$ be a unital subsemigroup of a discrete group $G$ and let $X$ be a product system over $P$.
Let $p_1, q_1, \dots, p_n, q_n \in P$ and $\eps, \eps' \in \{0,1\}$, and let $\xi_{p_i} \in X_{p_i}$ and $\xi_{q_i} \in X_{q_i}$ for $i = 1, \dots, n$.
Then for every $r \in P$ and $\xi_r \in X_r$ we have
\[
\left( \la_{p_1}(\xi_{p_1})^* \right)^\eps \la_{q_1}(\xi_{q_1}) \cdots
\la_{p_n}(\xi_{p_n})^* \la_{q_n}(\xi_{q_n})^{\eps'}\xi_r =
\begin{cases}
(\xi_{p_1}^*)^{\eps}\xi_{q_1}\cdots \xi_{p_n}^*\xi_{q_n}^{\eps'}\xi_r
& \text{if } r\in \Bx, \\
0 & \text{if } r \not\in \Bx,
\end{cases}
\]
where $\Bx := q_n^{-\eps'} p_n \dots q_1^{-1} p_1^{\eps} P$.
\end{proposition}

\subsection{The $\bo{K}$-cores and the $\bo{B}$-cores}
We will now define the $\bo{K}$-cores and the $\bo{B}$-cores for a representation.
See \cite{KP25} for more details. Let $P$ be a unital subsemigroup of a discrete group $G$ and let $X$ be a product system over $P$.
Let $t$ be a representation of $X$.
For $\Bx \in \J$, the \emph{$\bo{K}$-core on $\Bx$} of $\ca(t)$ is the closed linear span $\bo{K}_{\Bx, t_\ast}$ of the spaces
\[
\left( t_{p_1}(X_{p_1})^* \right)^{\eps} t_{q_1}(X_{q_1}) \cdots
t_{p_n}(X_{p_n})^* t_{q_n}(X_{q_n})^{\eps'}
\]
over all $n\in \bN$, $p_1, q_1, \dots, p_n, q_n \in P$ and $\eps, \eps' \in \{0,1\}$ that satisfy
\[
p_1^{-\eps} q_1 \cdots p_n^{-1} q_n^{\eps'} = e_P
\qand
q_n^{-\eps'} p_n \dots q_1^{-1} p_1^{\eps} P = \Bx.
\]
We do not assume that $\bo{K}_{\mt, t_\ast} = (0)$.
Each $\bo{K}_{\Bx, t_\ast}$ is a selfadjoint space, and the computation preceding \cite[Proposition 2.5]{KP25} with \eqref{Eq:faintersect} gives
\[
\bo{K}_{\Bx, t_\ast} \cdot \bo{K}_{\By, t_\ast}
\subseteq
\bo{K}_{\Bx \cap \By, t_\ast}
\foral \Bx, \By \in \J,
\]
so that each $\bo{K}_{\Bx, t_\ast}$ is a C*-subalgebra of $\ca(t)$.
We record the following remark for future use.

\begin{remark}\label{R:ideal}
For a finite $\cap$-closed $\F \subseteq \J$ we set 
\[
\bo{B}_{\F, t_\ast} := \sum_{\Bx \in \F} \bo{K}_{\Bx, t_\ast},
\]
which is again a C*-subalgebra of $\ca(t)$.
Indeed, let $\F \subseteq \J$ be finite and $\cap$-closed.
Since $\F$ is finite we may pick $\Bx \in \F$ to be maximal in $\F$ with respect to inclusion.
Maximality of $\Bx$ implies that $\F' := \F \setminus \{\Bx\}$ is also $\cap$-closed and that $\Bx \cap \By \in \F'$ for every $\By \in \F'$.
Using $\bo{K}_{\Bx,t_\ast} \cdot \bo{K}_{\By,t_\ast}\subseteq \bo{K}_{\Bx\cap \By,t_\ast}$ we get that $\bo{B}_{\F', t_\ast}$ is an ideal in $\bo{B}_{\F,t_\ast}$.
Hence an induction over the size of $\F$ gives that $\bo{B}_{\F, t_\ast}$ is closed as the sum of an ideal and a C*-subalgebra.

If $t$ is equivariant, then the fixed point algebra $[\ca(t)]_{e_P}$ is the inductive limit of the $\bo{B}_{\F, t_\ast}$ over the finite $\cap$-closed $\mt\in \F \subseteq \J$.
\end{remark}

We next observe that in the presence of non-degeneracy the parameters $\eps, \eps'$ may be dispensed with in the description of the $\bo{K}$-cores.

\begin{proposition} \label{P:fanoeps}
Let $P$ be a unital subsemigroup of a discrete group $G$ and let $X$ be a non-degenerate product system over $P$.
Let $t$ be a representation of $X$.
Then for every $\Bx \in \J$ the core $\bo{K}_{\Bx, t_\ast}$ is the closed linear span of the spaces
\[
t_{p_1}(X_{p_1})^* t_{q_1}(X_{q_1}) \cdots
t_{p_n}(X_{p_n})^* t_{q_n}(X_{q_n})
\]
over all $n \in \bN$ and $p_1, q_1, \dots, p_n, q_n \in P$ with
$p_1^{-1} q_1 \cdots p_n^{-1} q_n = e_P$ and
$q_n^{-1} p_n \dots q_1^{-1} p_1 P = \Bx$.
\end{proposition}

\begin{proof}
Non-degeneracy gives $[A \cdot X_q] = X_q$, and hence $A \cdot X_q = X_q$ by the Cohen Factorisation Theorem.
Let $\xi_q \in X_q$ and write $\xi_q = a \xi_q'$ with $a \in A$ and $\xi_q' \in X_q$.
Then
\[
t_{e_P}(a^*)^* t_q(\xi_q') = t_q(a \xi_q') = t_q(\xi_q)
\qand
t_q(\xi_q)^* = t_q(\xi_q')^* t_{e_P}(a^*).
\]
Applying the first identity to the leftmost factor of a word with $\eps = 0$ produces a word with $\eps = 1$ and $p_1 = e_P$, and applying the second to the rightmost factor of a word with $\eps' = 0$ produces a word with $\eps' = 1$ and $q_n = e_P$.
Neither operation changes the group element $p_1^{-\eps} q_1 \cdots p_n^{-1} q_n^{\eps'}$ nor the constructible ideal $q_n^{-\eps'} p_n \dots q_1^{-1} p_1^{\eps} P$, as $e_P$ acts trivially in both expressions.
\end{proof}

We will use the following immediate consequence of Proposition~\ref{P:fcmt}.

\begin{proposition} \label{P:fcmt2} \cite[Proposition 2.6]{KP25}
Let $P$ be a unital subsemigroup of a discrete group $G$, let $X$ be a product system over $P$ and let $\wh{t}$ be the universal representation of $X$ with $\T(X) = \ca(\wh{t})$.
Then for every $\Bx \in \J$ and every $r \notin \Bx$ we have
\[
\la_\ast(b_{\Bx}) \xi_r = 0
\foral b_\Bx \in \bo{K}_{\Bx, \wh{t}_\ast}, \xi_r \in X_r.
\]
\end{proposition}

\subsection{Fock covariance}
Let $P$ be a unital subsemigroup of a discrete group $G$ and let $X$ be a product system over $P$.
Following \cite{KP25}, consider the ideal 
\[
\J_{\cov}^{\fock} := \sca{\ker \la_\ast \cap [\T(X)]_{e_P}} \lhd \T(X),
\]
and write $\T_{\cov}^{\fock}(X)$ for the corresponding quotient, which inherits the coaction by $G$ from $\T(X)$.
The \emph{Fock covariant bundle} of $X$ is the Fell bundle $\F\C X := \{ [\T_{\cov}^{\fock}(X)]_g \}_{g \in G}$, and a representation of $X$ that promotes to a representation of $\F\C X$ is called a \emph{Fock covariant representation of $X$}.
Trivially $\T_{\cov}^{\fock}(X)$ and $\T_\la(X)$ arise from Fock covariant representations, and moreover we have the canonical $*$-isomorphisms
\[
\T_{\cov}^{\fock}(X) \simeq \ca_{\max}(\F\C X)
\qand
\T_\la(X) \simeq \ca_\la(\F\C X),
\]
where $\ca_{\max}(\F\C X)$ and $\ca_\la(\F\C X)$ denote the full cross-sectional C*-algebra and the reduced cross-sectional C*-algebra of $\F\C X$, see \cite[Definitions 16.25 and 17.6]{Exe17} and \cite{DKKLL22, KP25} for a detailed proof of these facts.

In the special case that $P$ is a unital subsemigroup of an amenable discrete group $G$, \cite[Theorem 20.7 and Proposition 20.2]{Exe17} yields a canonical $*$-isomorphism $\T_{\cov}^{\fock}(X) \simeq \T_\la(X)$.
Therefore the universal property passes to $\T_\la(X)$ in the sense that if $t$ is a Fock covariant representation of $X$, then there exists a canonical $*$-epimorphism $\T_\la(X) \to \ca(t)$.

Regarding the $\bo{K}$-cores, by Proposition~\ref{P:fcmt} we have
\begin{equation} \label{Eq:faKempty}
\bo{K}_{\mt, \la_\ast} = (0).
\end{equation}
More generally, in \cite[Theorem 3.2]{KP25} it is shown that if $t$ is a Fock covariant representation, then $\bo{K}_{\mt,t_\ast}=(0)$ as well.
Injective Fock covariant representations admit an intrinsic characterisation in terms of the $\bo{K}$-cores, see \cite[Theorem 3.2]{KP25}.
We shall need the following consequence of Fock covariance, recorded here for future use.

\begin{proposition}\label{P:Kcoresiso}
\cite[Remark 3.3]{KP25}
Let $P$ be a unital subsemigroup of a discrete group $G$ and let $X$ be a product system over $P$.
Suppose that $t$ is an injective Fock covariant representation of $X$.
Then the canonical quotient map $\T_{\cov}^{\fock}(X) \to \ca(t)$ restricts to a $*$-isomorphism on each $\bo{K}$-core.
\end{proposition}

\begin{remark}\label{R:scr}
Let $P$ be a unital subsemigroup of a discrete group $G$ and let $X$ be a product system over $P$. 
We write $A\times_X P$ for the universal boundary quotient of $\T(X)$ constructed by Sehnem in \cite[Theorem~3.10]{Seh19}. 
Every strong covariant representation of $X$ in the sense of \cite[Definition~3.2]{Seh19} is Fock covariant, since the strong covariance relations are graded and are induced from $\T_\la(X)$, see \cite[Section~5]{DKKLL22}. 
In \cite[Lemma~3.4]{Seh19} it is shown that $A\times_X P$ is an equivariant quotient of $\T(X)$, i.e., it inherits a coaction by $G$.
As observed in \cite[Section~5]{DKKLL22} this gives rise to a Fell bundle $\S\C X$ over $G$, and $A\times_X P$ is the full cross-sectional C*-algebra of $\S\C X$. 
We write $A\times_{X,\la} P$ for the corresponding reduced cross-sectional C*-algebra, which inherits a normal coaction by $G$. 
Since $A\times_{X,\la} P$ is a quotient of $A\times_X P$, every $*$-representation of either algebra arises from a Fock covariant representation of $X$. Consequently, all results we obtain for Fock covariant representations also apply to the boundary quotients $A\times_X P$ and $A\times_{X,\la} P$.
\end{remark}

We close this subsection with the following example that indicates how semigroup C*-algebras are captured by C*-algebras associated to product systems.
For a more detailed description see \cite[Subsection 3.3]{KP25}.

\begin{example}\label{E:semig}
Let $P$ be a unital subsemigroup of a discrete group $G$.
A map $w\colon P \to \B(H)$ is called an \emph{isometric representation of $P$} if it is a semigroup homomorphism and each $w_p$ is an isometry.
Following the notation of \cite{LS22}, for a word of even length we write $\al = (p_1, q_1, \dots, p_n, q_n)$ where $p_1, q_1, \dots, p_n, q_n\in P$.
A word $\al = (p_1, q_1, \dots, p_n, q_n)$ is called \emph{neutral} if $p_1^{-1} q_1 \cdots p_n^{-1} q_n = e_P$.
We also write
\[
K(\al) := q_n^{-1} p_n \dots q_1^{-1} p_1 P
\]
for the induced constructible ideal in $\J$, and
\[
\dot{w}_\al := w_{p_1}^* w_{q_1} \cdots w_{p_n}^* w_{q_n} \in \B(H).
\]

Let $\T(P):=\ca(\{v_p \mid p\in P\})$ denote the universal C*-algebra with respect to isometric representations of $P$.
Moreover, we write $\ca_\la(P)$ for the reduced C*-algebra of the semigroup which is generated by the left regular representation $V\colon P \to \B(\ell^2(P))$.
In \cite{LS22} the exact algebraic relations that describe Fock covariance are identified.
Following the notation of \cite{LS22} we denote by $\T_u(P)$ the universal C*-algebra with respect to (isometric) representations $w \colon P \to \B(H)$ that satisfy the relations:
\begin{itemize}
\item [\rm{(T1)}] $w_{e_P}=1$;
\item [\rm{(T2)}] $\dot{w}_{\al}=0$ if $K(\al)=\mt$ for a neutral word $\al$;
\item [\rm{(T3)}] $\dot{w}_{\al}=\dot{w}_{\be}$ if $K(\al)=K(\be)$ for neutral words $\al$ and $\be$; and
\item [\rm{(T4)}] $\prod_{\be\in F}(\dot{w}_{\al}-\dot{w}_{\be})=0$ whenever $\al$ is a neutral word, $F$ is a finite set of neutral words and $K(\al)=\bigcup_{\be\in F}K(\be)$.
\end{itemize}
These relations are of course satisfied by the left regular representation $V$ of $P$.

We define the concrete product system $X_P$ in $\T(P)$ by
\[
X_{P, p} := \bC v_p \foral p \in P.
\]
Then the association
\[
t = \{t_p\}_{p \in P} \mapsto w_t = \{t_p(v_p)\}_{p \in P}
\]
defines a bijection between unital (Fock covariant) representations of $X_P$ and isometric representations of $P$ (resp.\ that satisfy conditions (T1)--(T4)).
This bijection matches the Fock representation $\la$ of $X_P$ to the left regular representation $V$ of $P$.
See \cite[Proposition 3.9]{KP25} for a detailed proof.
In particular, we obtain the canonical $*$-isomorphisms
\[
\T(X_P)\simeq \T(P), \quad \T_\la(X_P)\simeq \ca_\la(P) \qand\T_{\cov}^{\fock}(X_P)\simeq \T_u(P).
\]
Hence, if $G$ is assumed amenable, then we get $\ca_\la(P) \simeq \T_u(P)$.

Moreover, let $w$ be an isometric representation of $P$ that satisfies conditions (T1)--(T4) and let $t$ be its corresponding Fock covariant representation of $X_P$.
Let $\Bx \in \J$ be a constructible ideal.
If $\Bx=\mt$, then condition (T2) implies that $\bo{K}_{\mt,t_\ast}=(0)$.
If $\Bx\neq \mt$, then pick a neutral word $\al=(p_1,q_1,\dots,p_n,q_n)$ such that $\Bx = K(\al)$.
In this case, condition (T3) implies that 
\[
\bo{K}_{\Bx,t_\ast} = \bC \dot{w}_\al.
\]
\end{example}

\subsection{Fowler's product systems} \label{Ss:Fow}

Let $P$ be a unital subsemigroup of a discrete group $G$.
Following \cite[Definition 2.1]{Fow02}, a \emph{product system in the sense of Fowler} over $P$ with coefficients in a C*-algebra $A$ is a family $\{X_p\}_{p \in P}$ of C*-correspondences over $A$, with left actions implemented by the $*$-homomorphisms 
\[
\{\vphi_p\colon A \to \L(X_p)\}_{p\in P},
\]
together with a family of maps 
\[
\{u_{p,q} \colon X_p \otimes_A X_q \to X_{pq}\}_{p, q \in P},
\]
such that:
\begin{enumerate} 
\item $X_{e_P} = A$;
\item $u_{p,q}$ is a unitary for all $p,q \in P\setminus \{e_P\}$;
\item $u_{e_P, q}$ and $u_{p, e_P}$ are given by the left and the right action of $A$; and 
\item the family $\{u_{p,q}\}_{p,q \in P}$ is associative, i.e.,
\[
u_{p, qr} \circ (\id_{X_p} \otimes u_{q,r})
=
u_{pq, r} \circ (u_{p,q} \otimes \id_{X_r})
\foral p, q, r \in P.
\]
\end{enumerate}
We note that as the right action on a C*-correspondence is always non-degenerate, the maps $u_{p, e_P}$ are unitaries as well.
We call $X$ as above \emph{non-degenerate} if the $*$-homomorphisms $\vphi_p$ are non-degenerate for every $p\in P$, and hence $u_{e_P,p}$ is a unitary for every $p\in P$.
When $P$ contains non-trivial invertible elements then the $u_{p,q}$ being unitaries for all $p,q \in P \setminus \{e_P\}$ implies that so are the $u_{e_P,q}$ for every $q \in P$, see \cite[Remark 1.3]{KL19b}.
Therefore it has become customary in the literature to consider non-degenerate product systems.

If $X$ is a saturated concrete product system, then $X$ defines a non-degenerate product system in the sense of Fowler by means of the unitaries
\[
u_{p,q} \colon X_p \otimes_A X_q \to X_{pq};\
\xi_p \otimes \xi_q \mapsto \xi_p \cdot \xi_q.
\]
Conversely, every non-degenerate product system in the sense of Fowler is unitarily equivalent to a saturated concrete one, namely to $\la(X)$, see \cite[Subsection 2.5]{KP25}.
In this work, we focus on saturated product systems.
Depending on what is most convenient, we view them either concretely or abstractly in the sense of Fowler.
We use both viewpoints since concrete models make proofs easier to follow with less notation, while the abstract approach makes it easier to track left actions by adjointable operators.

A saturated product system gives rise to the $*$-homomorphisms
\[
i_p^{pq} \colon \L(X_p) \to \L(X_{pq});\
i_p^{pq}(S) := u_{p,q} \left( S \otimes \id_{X_q} \right) u_{p,q}^*.
\]
Likewise for $p, q, r \in P$ define the linear map
\[
i^r_{p,q} \colon \L(X_q, X_p) \to \L(X_{qr}, X_{pr});\
S \mapsto u_{p,r} \left( S \otimes \id_{X_r} \right) u_{q,r}^*.
\]
Note that $i_p^{pq}= i^q_{p,p}$ for every $p,q \in P$.

The following lemmas are well known, and we record them here for future use.

\begin{lemma} \label{L:faj}
Let $P$ be a unital subsemigroup of a discrete group $G$ and let $X$ be a saturated product system over $P$.
For all $p, q, s, r, r_1, r_2 \in P$ the following hold:
\begin{enumerate}
\item $i^r_{p,q}$ is linear and contractive, and
$i^r_{p,q}(S)^* = i^r_{q,p}(S^*)$ for $S \in \L(X_q, X_p)$.
\item $i^r_{p,q}(S) \, i^r_{q,s}(T) = i^r_{p,s}(S T)$ for $S \in \L(X_q, X_p)$ and $T \in \L(X_s, X_q)$.
\item $i^{r_1 r_2}_{p,q}(S) = i^{r_2}_{p r_1, q r_1} \circ i^{r_1}_{p,q}(S)$ for $S \in \L(X_q, X_p)$.
\item $i^r_{p,q}(T S) = i_p^{pr}(T) \, i^r_{p,q}(S)$ and $i^r_{p,q}(S T') = i^r_{p,q}(S) \, i_q^{qr}(T')$, for $T \in \L(X_p)$, $T' \in \L(X_q)$ and $S \in \L(X_q, X_p)$.
\end{enumerate}
\end{lemma}

\begin{lemma} \label{L:famult}
Let $P$ be a unital subsemigroup of a discrete group $G$ and let $X$ be a saturated product system over $P$.
For $\xi_p \in X_p$, $p \in P$, let $M^{e_P, p}_{\xi_p} \colon A \to X_p$ be the operator given by $M^{e_P, p}_{\xi_p} a := \xi_p a$ for $a \in A$.
Then $M^{e_P, p}_{\xi_p} \in \K(A, X_p)$.
\end{lemma}

\begin{lemma} \label{L:facreate}
Let $P$ be a unital subsemigroup of a discrete group $G$ and let $X$ be a saturated product system over $P$.
Then for $\xi_p\in X_p$, $p \in P$, we have
\[
\la_p(\xi_p)|_{X_r} = i^r_{p, e_P} \left( M^{e_P, p}_{\xi_p} \right)
\qand
\la_p(\xi_p)^*|_{X_{pr}}
= i^r_{e_P, p} \left( \left( M^{e_P, p}_{\xi_p} \right)^* \right).
\]
\end{lemma}

Recall that amplifying a compact operator does not produce a compact operator in general; see Remark~\ref{R:faneed}.
It does, however, when the amplification is along an invertible element of the semigroup.

\begin{lemma} \label{L:faunits}
Let $P$ be a unital subsemigroup of a discrete group $G$ and let $X$ be a saturated product system over $P$.
If $z \in P^*$, then the following hold:
\begin{enumerate}
\item $\vphi_z(A) \subseteq \K(X_z)$;
\item $i^z_{p,q}\left( \K(X_q, X_p) \right) \subseteq \K(X_{qz}, X_{pz})$ for
all $p, q \in P$.
\end{enumerate}
\end{lemma}

\begin{proof}
(i) If $z=e_P$, then the result follows trivially.
Suppose that $z\neq e_P$.
We have 
\[
X_z \otimes_A X_{z^{-1}} \simeq X_{e_P} = A,
\]
so it suffices to show this for the elements $a := u_{z, z^{-1}}(\xi_z \otimes \eta_{z^{-1}})$ with $\xi_z \in X_z$ and $\eta_{z^{-1}} \in X_{z^{-1}}$.
For $\ze_z \in X_z$ we have $\eta_{z^{-1}} \otimes \ze_z \in X_{z^{-1}} \otimes_A X_z \simeq A$, and therefore
\begin{align*}
\vphi_{z}(a) \ze_z
& =
u_{e_P,z} \circ (u_{z,z^{-1}} \otimes \id_{X_{z}})((\xi_z \otimes \eta_{z^{-1}}) \otimes \ze_z) \\
& =
u_{z, e_P} \circ (\id_{X_z} \otimes u_{z^{-1},z})(\xi_z \otimes (\eta_{z^{-1}} \otimes \ze_z)) \\
& =
M^{e_P, z}_{\xi_z} \left( M^{z, e_P}_{\eta_{z^{-1}}} \ze_z \right).
\end{align*}
Thus $\vphi_{z}(a) = M^{e_P,z}_{\xi_z} \circ M^{z, e_P}_{\eta_{z^{-1}}}$, which lies in $\K(X_z)$ since $M^{e_P,z}_{\xi_z} \in \K(A, X_z)$ by Lemma~\ref{L:famult} and $M^{z,e_P}_{\eta_{z^{-1}}} \in \L(X_z, A)$.
The claim follows by linearity and continuity.

\smallskip
\noindent
(ii) It suffices to consider $S = \theta^{X_q, X_p}_{\xi_p, \eta_q}$ with $\xi_p \in X_p$ and $\eta_q \in X_q$.
Since $[X_q \cdot A] = X_q$, the Cohen Factorisation Theorem allows us to write $\eta_q = \eta'_q a$ with $\eta'_q \in X_q$ and $a \in A$.
For $\ze_q \in X_q$ and $\ka_z \in X_z$ we then have
\begin{align*}
i^z_{p,q}(S) \left( u_{q,z}(\ze_q \otimes \ka_z) \right)
&=
u_{p,z} \left( \xi_p \sca{\eta_q, \ze_q} \otimes \ka_z \right)\\
&=
u_{p,z} \left( \xi_p \otimes
\vphi_{z}(a^*) \vphi_{z}(\sca{\eta'_q, \ze_q}) \ka_z \right)\\
&=
i^z_{p, e_P} \left( M^{e_P,p}_{\xi_p} \right) \;
\vphi_{z}(a^*) \;
i^z_{e_P, q} \left( \left( M^{e_P,q}_{\eta'_q} \right)^* \right)\left( u_{q,z}(\ze_q \otimes \ka_z) \right),
\end{align*}
and hence
\[
i^z_{p,q}(S)
=
i^z_{p, e_P} \left( M^{e_P,p}_{\xi_p} \right) \;
\vphi_{z}(a^*) \;
i^z_{e_P, q} \left( \left( M^{e_P,q}_{\eta'_q} \right)^* \right).
\]
The outer two factors are adjointable and the middle one is compact by item (i), and thus $i^z_{p,q}(S) \in \K(X_{qz}, X_{pz})$. By linearity and continuity the proof is complete.
\end{proof}

For $r \in P$ we write $Q_r \colon \F X \to X_r$ for the canonical projection on the $r$-summand, and we consider the evaluation map 
\[
\ev_r \colon \L(\F X) \to \L(X_r);\
T \mapsto Q_r T|_{X_r}.
\]
We will need the following lemma.

\begin{lemma} \label{L:fadiag}
Let $P$ be a unital subsemigroup of a discrete group $G$ and let $X$ be a saturated product system over $P$.
If $\Bx \in \J$ and $b \in \bo{K}_{\Bx, \la_\ast}$, then the following hold:
\begin{enumerate}
\item $b X_r \subseteq X_r$ for $r \in \Bx$ and $b X_r = (0)$ for $r \notin \Bx$; in particular $b$ is determined by the family $(\ev_r(b))_{r \in \Bx}$, and each $\ev_r$ restricts to a $*$-homomorphism on $\bo{K}_{\Bx, \la_\ast}$.
\item $\ev_{rs}(b) = i_r^{rs}(\ev_r(b))$ for all $r \in \Bx$ and $s \in P$.
\end{enumerate}
\end{lemma}

\begin{proof}
For simplicity we show this in the concrete product system setting.
Item (i) follows by an application of Proposition~\ref{P:fcmt} and Proposition~\ref{P:fcmt2}.
Since $X_r$ reduces every element of $\bo{K}_{\Bx, \la_\ast}$, we get that $\ev_r$ is a $*$-homomorphism on the C*-algebra $\bo{K}_{\Bx, \la_\ast}$.

For item (ii), fix $s\in P$ and pick some $\xi_r \in X_r$ and $\eta_s \in X_s$.
We have
\[
\ev_{rs}(b)(\xi_r \eta_s) = b \left( \xi_r \eta_s \right)
=
( b \xi_r ) \eta_s
=
( \ev_r(b) \xi_r ) \eta_s=i_r^{rs}(\ev_r(b))(\xi_r\eta_s).
\]
Using $[X_r \cdot X_s] = X_{rs}$ completes the proof.
\end{proof}

We record how the compact operators sit inside the fibres of the Fock space.
If $t$ is a representation of a product system $X$, we write $t^{(p)}$ for the induced $*$-homomorphism on each $\K(X_p)$ such that 
\[
t^{(p)}(\theta^{X_p}_{\xi_p, \eta_p}) = t_p(\xi_p) t_p(\eta_p)^* \foral \xi_p, \eta_p \in X_p.
\]
By \cite[Proposition B.8]{Kat04} we have that if $t_{e_P}$ is a nuclear map, then so is $t^{(p)}$.

\begin{lemma} \label{L:facptfibre}
Let $P$ be a unital subsemigroup of a discrete group $G$ and let $X$ be a saturated product system over $P$.
Let $p \in P$ and $k \in \K(X_p)$.
Then $\la^{(p)}(k) X_r = (0)$ for $r \notin pP$, and $\la^{(p)}(k) X_r \subseteq X_r$ for $r \in pP$. Moreover,
\[
\ev_{pq} \left( \la^{(p)}(k) \right) = i_p^{pq}(k)
\foral q \in P.
\]
\end{lemma}

\begin{proof}
For simplicity we show this in the concrete product system setting.
By linearity and continuity we may assume that $k = \theta^{X_p}_{\xi, \eta}$, so that $\la^{(p)}(k) = \la_p(\xi) \la_p(\eta)^*$.
The first two claims follow from Proposition~\ref{P:fcmt}.
Using $[X_p \cdot X_q] =X_{pq}$, it suffices to evaluate on elements $\ze_p \ze_q'$ with $\ze_p \in X_p$ and $\ze_q' \in X_q$, where we get
\[
\ev_{pq}(\la^{(p)}(k)) (\ze_p \ze_q')
= Q_{pq} \la_p(\xi) \la_p(\eta)^* \ze_p \ze_q'
= \xi \eta^* \ze_p \ze_q'
= (\theta^{X_p}_{\xi, \eta} \ze_p ) \ze_q'
= i_p^{pq}(\theta^{X_p}_{\xi, \eta})(\ze_p \ze_q'),
\]
as required.
\end{proof}

\subsection{Exactness of the $\bo{K}$-cores} \label{Ss:exkcores}

Exactness for injective equivariant Fock covariant representations of a product system $X$ has been studied by many authors, see for example \cite{DS01, Kak20, KKLL23, Kat04, LS22}.
The following recovers an intermediate step (that is, item (i) of Question~\ref{Qn:main2}) in these results without additional hypotheses on $X$ or $P$, or the representation.

\begin{theorem}\label{T:Kcores}
Let $P$ be a unital subsemigroup of a discrete group $G$ and let $X$ be a product system over $P$.
Let $t$ be a representation of $X$.
Then the following are equivalent:
\begin{enumerate}
\item $\bo{K}_{\Bx,t_\ast}$ is exact for every $\Bx \in \J$.
\item $\bo{B}_{\F,t_\ast}$ is exact for every finite and $\cap$-closed $\F \subseteq \J$.
\item The inductive limit of the $\bo{B}_{\F, t_\ast}$ over the finite $\cap$-closed $\mt\in\F\subseteq\J$ is exact.
\end{enumerate}
\end{theorem}

\begin{proof}
\noindent
[(i) $\Rightarrow$ (ii)].
We proceed by induction on $|\F|$.
If $|\F|=1$, then $\bo{B}_{\F,t_\ast}=\bo{K}_{\Bx,t_\ast}$ for some $\Bx \in \F$ which is exact by our hypothesis.
Now suppose that $|\F|=k+1$ and that $\bo{B}_{\F',t_\ast}$ is exact for all finite and $\cap$-closed $\F'\subseteq \J$ with $|\F'|\leq k$.
Write $ \F' := \F \setminus \{\Bx\}$ with $\Bx$ maximal in $\F$.
By our hypothesis we have that $\bo{K}_{\Bx,t_\ast}$ is exact and by our inductive hypothesis we have that $\bo{B}_{\F',t_\ast}$ is also exact.
Moreover, by Remark~\ref{R:ideal} we have that $\bo{B}_{\F',t_\ast}$ is an ideal in $\bo{B}_{\F,t_\ast} = \bo{B}_{\F',t_\ast}+\bo{K}_{\Bx,t_\ast}$.
Theorem~\ref{T:Exsum} then implies that $\bo{B}_{\F,t_\ast}$ is also exact, as required.

\smallskip
\noindent
[(ii) $\Rightarrow$ (iii)].
This is immediate as exactness passes to inductive limits.

\smallskip
\noindent
[(iii) $\Rightarrow$ (i)].
This is immediate as exactness passes to C*-subalgebras.
\end{proof}

We obtain the following corollary.

\begin{corollary}\label{C:Excores-3}
Let $P$ be a unital subsemigroup of an exact discrete group $G$ and let $X$ be a product system over $P$.
Then the following are equivalent:
\begin{enumerate}
\item $\T_\la(X)$ is exact.
\item $\bo{K}_{\Bx,\la_\ast}$ is exact for every $\Bx \in \J$.
\item $\ca(t)$ is exact for a quotient $\ca(t)$ of $\T_\la(X)$ that is injective on $A$.
\item $\bo{K}_{\Bx,t_\ast}$ is exact for every $\Bx \in \J$, for a quotient $\ca(t)$ of $\T_\la(X)$ that is injective on $A$.
\item $\ca(t)$ is exact for an injective equivariant Fock covariant representation $t$ of $X$ that admits a normal coaction.
\item $\bo{K}_{\Bx,t_\ast}$ is exact for every $\Bx \in \J$, for an injective equivariant Fock covariant representation $t$ of $X$ that admits a normal coaction.
\end{enumerate}
If any of the items above holds, then items (iii)--(vi) hold for any representation $t$ with the properties of the corresponding item.
\end{corollary}

\begin{proof}
\noindent
[(i) $\Rightarrow$ (ii)].
Each $\bo{K}_{\Bx, \la_\ast}$ is a C*-subalgebra of $\T_\la(X)$ and exactness passes to C*-subalgebras.

\smallskip
\noindent
[(ii) $\Rightarrow$ (i)].
The Fock representation is equivariant, and thus the fixed point algebra $[\T_\la(X)]_{e_P}$ is the inductive limit of the $\bo{B}_{\F, \la_\ast}$ over the finite $\cap$-closed $\mt \in \F \subseteq \J$.
Hence Theorem~\ref{T:Kcores} yields that $[\T_\la(X)]_{e_P}$ is exact.
Exactness of the group and \cite[Proposition 25.12]{Exe17} (which applies due to normality of the coaction on $\T_\la(X)$) then give that $\T_\la(X)$ is exact.

\smallskip
\noindent
[(ii) $\Leftrightarrow$ (iv)].
This follows from Proposition~\ref{P:Kcoresiso} due to injectivity of $t$.

\smallskip
\noindent
[(ii) $\Leftrightarrow$ (vi)].
Same as [(ii) $\Leftrightarrow$ (iv)].

\smallskip
\noindent
[(i) $\Rightarrow$ (iii)].
This is immediate as exactness is preserved under quotients.

\smallskip
\noindent
[(iii) $\Rightarrow$ (iv)].
Every $\bo{K}_{\Bx, t_\ast}$ is a C*-subalgebra of $\ca(t)$ and exactness passes to C*-subalgebras.

\smallskip
\noindent
[(i) $\Rightarrow$ (v)].
If $t$ is an equivariant Fock covariant representation of $X$ that admits a normal coaction, then $\ca(t)$ is a quotient of $\T_\la(X)$, see for example \cite[comments prior to Proposition~2.1]{KP25}.
Hence this follows as exactness is preserved under quotients.

\smallskip
\noindent
[(v) $\Rightarrow$ (vi)].
Same as [(iii) $\Rightarrow$ (iv)].
\end{proof}

When the semigroup is embedded in an amenable discrete group we have the following two corollaries.

\begin{corollary}\label{C:Excores-1}
Let $P$ be a unital subsemigroup of an amenable discrete group $G$ and let $X$ be a product system over $P$.
Let $t$ be an equivariant representation of $X$.
Then the following are equivalent:
\begin{enumerate}
\item $\bo{K}_{\Bx,t_\ast}$ is exact for every $\Bx \in \J$.
\item $\bo{B}_{\F,t_\ast}$ is exact for every finite and $\cap$-closed $\F \subseteq \J$.
\item $[\ca(t)]_{e_P}$ is exact.
\item $\ca(t)$ is exact.
\end{enumerate}
\end{corollary}

\begin{proof}
\noindent
[(i) $\Leftrightarrow$ (ii) $\Leftrightarrow$ (iii)].
This follows by Theorem~\ref{T:Kcores} as equivariance implies that the inductive limit over the $\bo{B}$-cores is the fixed point algebra.

\smallskip
\noindent
[(iv) $\Rightarrow$ (iii)].
This follows as exactness passes to C*-subalgebras.

\smallskip
\noindent
[(iii) $\Rightarrow$ (iv)].
By \cite[Theorem 20.7 and Proposition 20.2]{Exe17} we have 
\[
\ca(t) \simeq \ca_\la(\{[\ca(t)]_g\}_g),
\]
and as amenable groups are exact applying \cite[Proposition 25.12]{Exe17} yields the required result.
\end{proof}

If $G$ is assumed to be amenable, then $\T_\la(X)$ is universal with respect to the Fock covariant representations of $X$ and by Proposition~\ref{P:Kcoresiso} and Corollary~\ref{C:Excores-3} we obtain the following corollary.

\begin{corollary}\label{C:Excores-2}
Let $P$ be a unital subsemigroup of an amenable discrete group $G$ and let $X$ be a product system over $P$.
Then $\T_\la(X)$ is exact if and only if $\ca(t)$ is exact for some (and thus for all) injective Fock covariant representation $t$ of $X$.
\end{corollary}

\begin{remark}
The use of Theorem~\ref{T:Kcores} allows us to remove the assumption of linear independence of the $\bo{K}$-cores from \cite[Theorem 5.6]{KLS25} for the product systems of \cite[Theorem 5.4]{KLS25}, and show that exactness of $A$ is equivalent to exactness of $\ca(t)$ for any injective Fock covariant representation $t$ in these results.
\end{remark}

As an application to product systems with one-dimensional fibres we retrieve the following corollary.

\begin{corollary}\label{C:exsem}
Let $P$ be a unital subsemigroup of an exact discrete group $G$.
Then the reduced C*-algebra $\ca_\la(P)$ of the semigroup $P$ is exact.
\end{corollary}

\begin{proof}
By Example~\ref{E:semig} we have that $\ca_\la(P)\simeq \T_\la(X_P)$, and that $\bo{K}_{\Bx, \la_\ast}$ is a finite-dimensional C*-algebra for every $\Bx\in \J$, and hence exact. Corollary~\ref{C:Excores-3} then yields that $\T_\la(X_P)$ is exact.
\end{proof}

\subsection{Nuclear embeddability of the $\bo{K}$-cores}

We now pass to nuclearity.
No assumption is made on the product system or on the semigroup beyond the standing hypothesis that $P$ sits in an amenable discrete group.
This extends the analogous results in \cite{Kak20, KKLL23, Kat04, LS22}.

\begin{theorem}\label{T:Nuccores}
Let $P$ be a unital subsemigroup of an amenable discrete group $G$ and let $X$ be a product system over $P$.
Let $t$ be an equivariant representation of $X$.
Then the following are equivalent:
\begin{enumerate}
\item $\bo{K}_{\Bx,t_\ast}$ is nuclearly embedded in $[\ca(t)]_{e_P}$ for every $\Bx \in \J$.
\item $\bo{B}_{\F,t_\ast}$ is nuclearly embedded in $[\ca(t)]_{e_P}$ for every finite and $\cap$-closed $\F \subseteq \J$.
\item $[\ca(t)]_{e_P}$ is nuclear.
\item $\ca(t)$ is nuclear.
\end{enumerate}
\end{theorem}

\begin{proof}
\noindent
[(i) $\Rightarrow$ (ii)].
Let $\F \subseteq \J$ be finite and $\cap$-closed.
We show that $\bo{B}_{\F,t_\ast}$ is nuclearly embedded in $[\ca(t)]_{e_P}$ by induction on $|\F|$.
If $|\F| = 1$, then $\bo{B}_{\F,t_\ast} = \bo{K}_{\Bx,t_\ast}$ for some $\Bx \in \F$, and the claim holds by hypothesis.
Now suppose that $|\F| = k+1$ and that the claim holds for all finite and $\cap$-closed families of cardinality at most $k$.
Write $\F'=\F \setminus \{\Bx\}$ with $\Bx$ maximal in $\F$, so that $\F'$ is $\cap$-closed and, by Remark~\ref{R:ideal}, the C*-algebra $\bo{B}_{\F',t_\ast}$ is an ideal in $\bo{B}_{\F,t_\ast} = \bo{B}_{\F',t_\ast} + \bo{K}_{\Bx,t_\ast}$.
By the inductive hypothesis $\bo{B}_{\F',t_\ast}$ is nuclearly embedded in $[\ca(t)]_{e_P}$, and by our hypothesis so is $\bo{K}_{\Bx,t_\ast}$.
An application of Theorem~\ref{T:Nucsum}
yields that $\bo{B}_{\F,t_\ast}$ is nuclearly embedded in $[\ca(t)]_{e_P}$, completing the induction.

\smallskip
\noindent
[(ii) $\Rightarrow$ (i)].
This is immediate by applying item (ii) to $\F = \{\Bx\}$.

\smallskip
\noindent
[(ii) $\Rightarrow$ (iii)].
The fixed point algebra $[\ca(t)]_{e_P}$ is the inductive limit of the C*-subalgebras $\bo{B}_{\F,t_\ast}$ over the $\cap$-closed and finite $\mt \in \F \subseteq \J$.
Since the inclusions $\bo{B}_{\F,t_\ast} \hookrightarrow [\ca(t)]_{e_P}$ are nuclear we obtain that $[\ca(t)]_{e_P}$ is nuclear.

\smallskip
\noindent
[(iii) $\Rightarrow$ (ii)].
This is immediate.

\smallskip
\noindent
[(iii) $\Rightarrow$ (iv)].
By \cite[Theorem 20.7 and Proposition 20.2]{Exe17} amenability of $G$ gives
\[
\ca(t) \simeq \ca_\la(\{[\ca(t)]_g\}_g) \simeq \ca_{\max}(\{[\ca(t)]_g\}_g),
\]
and the nuclear C*-algebra $[\ca(t)]_{e_P}$ is the unit fibre of this Fell bundle.
Hence $\ca(t)$ is nuclear by \cite[Proposition~25.10]{Exe17}.

\smallskip
\noindent
[(iv) $\Rightarrow$ (iii)].
Let $\de$ be the coaction by $G$ on $\ca(t)$ and let $\E_\de\colon \ca(t) \to [\ca(t)]_{e_P}$ be the induced conditional expectation.
If $\ca(t)$ is nuclear, then the inclusion $\io \colon [\ca(t)]_{e_P} \to \ca(t)$ is nuclear, and therefore so is $\id_{[\ca(t)]_{e_P}} = \E_\de \circ \io$,
showing that $[\ca(t)]_{e_P}$ is nuclear.
\end{proof}

We now obtain a direct proof of the fact that the reduced C*-algebra of a unital subsemigroup of an amenable discrete group is nuclear.
This is done in \cite[Corollary 4.10]{LS22} with the use of partial crossed products. See also \cite[Theorem 5.6.44]{CELY17}.

\begin{corollary}\cite[Theorem 5.6.44]{CELY17}, \cite[Corollary 4.10]{LS22}\label{C:nucsem}
Let $P$ be a unital subsemigroup of an amenable discrete group $G$.
Then $\ca_\la(P)$ is nuclear.
\end{corollary}

\begin{proof}
By Example~\ref{E:semig} we have $\ca_\la(P)\simeq \T_\la(X_P)$ and that $\bo{K}_{\Bx, \la_\ast}$ is a finite-dimensional C*-algebra for every $\Bx \in \J$, and hence nuclear.
Theorem~\ref{T:Nuccores} then yields that $\T_\la(X_P)$ is nuclear.
\end{proof}

\begin{remark}
Unlike Corollary~\ref{C:exsem}, we recall that Corollary~\ref{C:nucsem} does not hold if $G$ is just assumed exact.
For an example, let $G = P = \bF_2$ and consider the product system $X$ over $P$ with $X_p = \bC$ for all $p \in P$.
Every non-empty constructible ideal of a group equals the group itself, so that $\J = \{\mt, P\}$, and the Fock representation gives
\[
\T_\la(X) = \ca_\la(\bF_2),
\quad
\bo{K}_{\mt,\la_\ast}=(0) \qand
\bo{K}_{P, \la_\ast} = [\T_\la(X)]_{e_P} = \bC.
\]
While the $\bo{K}$-cores are nuclear, $\ca_\la(\bF_2)$ is not nuclear.
We note that $\ca_\la(\bF_2)$ is exact, so this example is consistent with Corollary~\ref{C:Excores-3}.

Moreover the converse of Corollary~\ref{C:nucsem} does not hold, i.e., it may be that $\ca_\la(P)$ is nuclear but $P$ cannot be embedded in any amenable group.
A counterexample was provided recently by Dor-On--Li \cite{DL26}.
\end{remark}

\section{Compactly aligned product systems}\label{S:capsfa}

We next consider the problem of obtaining exactness or nuclear embeddability of the $\bo{K}$-cores from that of the coefficient C*-algebra $A$.
We obtain such results with additional hypotheses on the semigroup and the structure of the product systems.

\subsection{Finitely aligned semigroups}

Let $P$ be a unital subsemigroup of a discrete group $G$.
A right ideal $\Bx$ is called \emph{finitely generated} if there are $p_1, \dots, p_m \in P$ with $\Bx = \bigcup_{i=1}^m p_i P$.
We regard the empty ideal as finitely generated by the empty family.

A generating family $\{p_1, \dots, p_m\}$ is called \emph{irredundant} if $p_i P \not\subseteq p_j P$ whenever $i \neq j$.
Every finite generating family contains an irredundant one, obtained by discarding redundant elements one at a time.

We gather below some elementary results from \cite{CG21, LV20, Li23, Spi12, Spi14}.
For the convenience of the reader we provide the proofs.

\begin{lemma} \label{L:famax}
Let $P$ be a unital subsemigroup of a discrete group $G$.
Let $\Bx$ be a right ideal of $P$ and let $\{p_1, \dots, p_m\}$ be an irredundant generating family for $\Bx$.
If $q \in \Bx$ satisfies $p_i \in q P$ for some $i$, then $q P = p_i P$ and $p_i = q x$ for a unique $x \in P^*$.
\end{lemma}

\begin{proof}
Since $q \in \Bx$ we have $q \in p_j P$ for some $j$, and therefore $p_i P \subseteq q P \subseteq p_j P$.
Irredundancy forces $j = i$, and thus $q P = p_i P$.
Write $p_i = q x$ and $q = p_i y$ with $x, y \in P$.
Then 
\[
p_i e_P = p_i = p_i (y x),
\]
and left cancellation gives $y x = e_P$.
Since this identity takes place in $G$ we get $x^{-1} = y \in P$, and hence $x \in P^*$.
Uniqueness follows from left cancellation.
\end{proof}

We will use the following lemma.

\begin{lemma} \label{L:fairr}
Let $P$ be a unital subsemigroup of a discrete group $G$.
Let $\Bx$ be a right ideal of $P$ admitting a finite generating family, and let 
\[
\{s_1, \dots, s_m\} \AND \{q_1, \dots, q_k\}
\]
be irredundant generating families for $\Bx$.
Then $m = k$ and there is a bijection $\si$ of $\{1, \dots, m\}$ and elements $x_i \in P^*$ such that $s_i = q_{\si(i)} x_i$ for all $i$.
Consequently, $\Bx$ is a disjoint union of finitely many principal right ideals if and only if $s_i P \cap s_j P = \mt$ for all $i \neq j$.
\end{lemma}

\begin{proof}
Since $s_i \in \Bx = \bigcup_j q_j P$ we may pick $\si(i)$ with $s_i \in q_{\si(i)} P$, and likewise we may pick $\tau(j)$ with $q_j \in s_{\tau(j)} P$.
Then
\[
s_i P \subseteq q_{\si(i)} P \subseteq s_{\tau \si (i)} P,
\]
and irredundancy of $\{s_1, \dots, s_m\}$ forces $\tau \si (i) = i$;
symmetrically $\si \tau (j) = j$.
Hence $\si$ and $\tau$ are mutually inverse bijections and $m = k$.
Moreover $s_i P = q_{\si(i)} P$, and Lemma~\ref{L:famax} applied to $q := q_{\si(i)} \in \Bx$ and the irredundant family $\{s_1, \dots, s_m\}$ gives $s_i = q_{\si(i)} x_i$ with $x_i \in P^*$.

For the last statement, suppose $\Bx = \bigcup_{j=1}^k q_j P$ with the union disjoint.
Then $\{q_1, \dots, q_k\}$ is irredundant.
Indeed, if $q_j P \subseteq q_\ell P$ with $j \neq \ell$, then $q_j \in q_j P \cap q_\ell P = \mt$, which is absurd.
By the first part $s_i P = q_{\si(i)} P$ for all $i$, and therefore the family $\{ s_i P \}_{i=1}^m$ is pairwise disjoint.
The converse is trivial.
\end{proof}

We obtain the following proposition.

\begin{proposition} \label{P:fafibres}
Let $P$ be a unital subsemigroup of a discrete group $G$ and let $X$ be a saturated product system over $P$.
If $\mt \neq \Bx \in \J$ is a finitely generated constructible ideal with irredundant generating family $\{p_1, \dots, p_m\}$, then the map
\[
\Phi_{\Bx} \colon \bo{K}_{\Bx, \la_\ast} \to \bigoplus_{i=1}^m \L(X_{p_i});\
b \mapsto \left( \ev_{p_1}(b), \dots, \ev_{p_m}(b) \right),
\]
is an injective $*$-homomorphism.
\end{proposition}

\begin{proof}
By item (i) of Lemma~\ref{L:fadiag} we have that $\Phi_\Bx$ is a $*$-homomorphism.
Suppose that $\Phi_\Bx(b) = 0$ and let $r$ be in $\Bx$.
Then $r = p_i s$ for some $i$ and some $s \in P$, and item (ii) of Lemma~\ref{L:fadiag} gives 
\[
\ev_r(b)=\ev_{p_i s}(b) = i_{p_i}^{p_i s}(\ev_{p_i}(b)) = 0.
\]
By item (i) of Lemma~\ref{L:fadiag} we conclude that $b = 0$.
\end{proof}

We next pass to finitely aligned semigroups.

\begin{definition} \label{D:fa}
A unital subsemigroup $P$ of a discrete group $G$ is called \emph{finitely aligned} if $pP \cap qP$ is a finitely generated right ideal for all $p, q \in P$.
\end{definition}

As $P$ is unital, this is the \emph{right ideal Howson property} of \cite{CG21} and the finite alignment of \cite[Section 8]{ES18}.
Spielberg \cite{Spi12, Spi14, Spi20} identified the finite alignment property for categories of paths and, more generally, for left cancellative small categories.
This class contains the right LCM semigroups, which are the \emph{singly aligned} semigroups.

\begin{example} \label{Ex:fa}
Finitely aligned semigroups that are not right LCM semigroups are abundant.
For example, the numerical semigroup $P = \{0, 2, 3, 4, \dots\} \subseteq \bZ$ is finitely aligned, while
\[
(2 + P) \cap (3 + P) = \{5, 6, 7, \dots\} = (5+P) \cup (6+P)
\]
is not principal, and hence $P$ is not a right LCM semigroup.
\end{example}

One of the key properties of finite alignment is that it implies finite generation for constructible ideals.
This is known, see for example \cite{Li23}, but we include a proof in our setting for the convenience of the reader.

\begin{proposition}\label{P:fajfg}
Let $P$ be a unital subsemigroup of a discrete group $G$.
Then the following are equivalent:
\begin{enumerate}
\item $P$ is finitely aligned.
\item Every constructible ideal $\Bx\in \J$ is finitely generated.
\end{enumerate}
\end{proposition}

\begin{proof}
\noindent
[(ii) $\Rightarrow$ (i)].
This is immediate, since $pP \cap qP \in \J$.

\smallskip
\noindent
[(i) $\Rightarrow$ (ii)].
Write $\J_{\rm fg}$ for the family of finitely generated right ideals of $P$.
Clearly $P = e_P P \in \J_{\rm fg}$ and $\mt \in \J_{\rm fg}$.
Since $\J$ is the smallest family of right ideals of $P$ that contains $P$ and $\mt$ and is closed under the operations $\Bx \mapsto p \Bx$ and $\Bx \mapsto p^{-1}\Bx$ for $p \in P$, and under finite intersections, it suffices to check that $\J_{\rm fg}$ shares the same closure properties.

Fix $\Bx = \bigcup_{i=1}^m r_i P \in \J_{\rm fg}$ and $p \in P$.
Then $p \Bx = \bigcup_{i=1}^m (p r_i) P \in \J_{\rm fg}$.
For the second operation, note that $p^{-1} \Bx = \bigcup_{i=1}^m p^{-1}(r_i P)$, and since the union of finitely generated right ideals is finitely generated, we may assume that $\Bx = r P$ is principal.
By finite alignment write $p P \cap r P = \bigcup_{j=1}^n w_j P$, and for each $j$ let $u_j \in P$ be the unique element with $w_j = p u_j$.
For $x \in P$ we have
\[
p x \in r P
\iff
p x \in pP \cap rP
\iff
p x \in p u_j P \text{ for some } j
\iff
x \in u_j P \text{ for some } j.
\]
Hence $p^{-1}(rP) = \bigcup_{j=1}^n u_j P \in \J_{\rm fg}$.
Finally, if $\Bx = \bigcup_{i=1}^m r_i P$ and $\By = \bigcup_{j=1}^n s_j P$ are in $\J_{\rm fg}$, then $\Bx \cap \By = \bigcup_{i, j} \left( r_i P \cap s_j P \right)$, and each $r_i P \cap s_j P$ is finitely generated by finite alignment.
Thus $\Bx \cap \By \in \J_{\rm fg}$, and the proof is complete.
\end{proof}

\subsection{Compactly aligned product systems over finitely aligned semigroups}

We now bring in compact alignment.
This additional property for a product system was identified by Fowler in the setting of quasi-lattice ordered groups \cite{Fow02}; for the product system of a higher-rank graph it is equivalent to finite alignment of the graph in the sense of Raeburn--Sims--Yeend \cite{RSY04}, see \cite[Theorem~5.4]{RS05}.
This condition allows for the Fock representation to have a tractable description, and in particular a Wick ordering.
There has been extensive work on compactly aligned product systems over quasi-lattice ordered groups or more generally over right LCM semigroups.

When $P$ is a right LCM semigroup, the intersection $pP \cap qP$ is either empty or principal.
Following the work of Kwa\'{s}niewski--Larsen \cite{KL19a, KL19b}, a saturated product system $X$ over a right LCM semigroup $P$ is called \emph{compactly aligned} if
\[
i_p^w(\K(X_p)) i_q^w(\K(X_q)) \subseteq \K(X_w),
\text{ if } 
 pP \cap qP = wP.
 \]
This is independent of the choice of $w$, see the discussion after \cite[Definition 2.9]{DKKLL22}.
In \cite[Proposition 3.6]{KP25} it is shown that for a compactly aligned saturated product system, an injective representation of $X$ is Fock covariant if and only if it is Nica covariant in the sense of \cite{KL19a, KL19b}.
In a somewhat converse direction, if the Fock representation of a saturated product system is Nica covariant, then the product system is compactly aligned, see \cite[Proposition 3.4]{KP25}.

For a finitely aligned $P$ the analogous requirement used in \cite{Spi20} is that the condition be imposed at each \emph{minimal} common right multiple.

\begin{definition} \label{D:mcm}
Let $P$ be a unital subsemigroup of a discrete group $G$ and let $p, q$ be in $P$.
An element $w \in pP \cap qP$ is a \emph{minimal common right multiple} of $p$ and $q$ if $w \in zP$ for $z \in pP \cap qP$ implies $w P = z P$ (and so $w=zx$ for some $x \in P^\ast$).
\end{definition}

The following extends the notion of compact alignment from the right LCM semigroups to finitely aligned semigroups.

\begin{definition} \label{D:famin}
Let $P$ be a unital finitely aligned subsemigroup of a discrete group $G$.
A saturated product system $X$ over $P$ is called \emph{compactly aligned} if
\[
i_p^w(k_p) \, i_q^w(k_q) \in \K(X_w)
\foral k_p \in \K(X_p), k_q \in \K(X_q),
\]
for all $p, q \in P$ and every minimal common right multiple $w$ of $p$ and $q$.
\end{definition}

Note that, if the left action on each fibre is by compacts, then compact alignment is automatically satisfied by a standard argument, see \cite[Proposition 4.7]{Lan95}.

\begin{remark} \label{R:famin}
By the discussion after \cite[Definition 2.9]{DKKLL22} we have that, if $w = zx$ for $x \in P^\ast$, then
\[
i_p^w(k_p) \, i_q^w(k_q) \in \K(X_w) 
\quad \Longleftrightarrow \quad
i_p^z(k_p) \, i_q^z(k_q) \in \K(X_z)
\]
for all $k_p \in \K(X_p)$ and $k_q \in \K(X_q)$.
Let $P$ be finitely aligned and let $\{w_1, \dots, w_m\}$ be an irredundant generating family for $pP \cap qP$.
By Lemma~\ref{L:famax} each $w_j$ is a minimal common right multiple of $p$ and $q$, and every element of $pP \cap qP$ lies in $w_j P$ for some $j$; hence every minimal common right multiple of $p$ and $q$ equals $w_j x$ for some $j$ and $x \in P^\ast$.
Therefore Definition~\ref{D:famin} reduces to the conditions at $w_1, \dots, w_m$, which by Lemma~\ref{L:fairr} do not depend on the choice of the irredundant family.
\end{remark}

The next lemma isolates the place where compact alignment enters.

\begin{lemma} \label{L:facomp}
Let $P$ be a unital finitely aligned subsemigroup of a discrete group $G$ and let $X$ be a compactly aligned saturated product system over $P$.
Let $u, v \in P$, let $w$ be a minimal common right multiple of $u$ and $v$, and write $w = u y_u = v y_v$ with $y_u, y_v \in P$.
Then
\[
i^{y_v}_{p,v} \left( \K(X_v, X_p) \right)
i^{y_u}_{u,q} \left( \K(X_q, X_u) \right)
\subseteq
\K \left( X_{q y_u}, X_{p y_v} \right)
\foral p, q \in P.
\]
\end{lemma}

\begin{proof}
Consider $k_1 \in \K(X_q, X_u)$ and $k_2 \in \K(X_v, X_p)$, and set
\[
S_1 := i^{y_u}_{u,q}(k_1) \in \L(X_{q y_u}, X_w)
\qand
S_2 := i^{y_v}_{p,v}(k_2) \in \L(X_w, X_{p y_v}),
\]
where we used that $u y_u = w = v y_v$.
Let $(e_\io)_\io$ and $(e'_\ka)_\ka$ be contractive approximate units for $\K(X_u)$ and $\K(X_v)$ respectively.
Then we have $\lim_\io e_\io k_1 = k_1$ and $\lim_\ka k_2 e'_\ka = k_2$ in norm, and therefore, by items (i) and (iv) of Lemma~\ref{L:faj}, we get
\[
S_1 = \lim_\io i^{y_u}_{u,q}(e_\io k_1) = \lim_\io i_u^{w}(e_\io) \, S_1
\qand
S_2 = \lim_\ka i^{y_v}_{p,v}(k_2 e'_\ka) = \lim_\ka S_2 \, i_v^{w}(e'_\ka).
\]
We thus obtain
\[
S_2 S_1 = \lim_\ka \lim_\io S_2 \, i_v^w(e'_\ka) \, i_u^w(e_\io) \, S_1.
\]
Since $w$ is a minimal common right multiple of $u$ and $v$, by compact alignment of $X$ we obtain that $i_v^w(e'_\ka) i_u^w(e_\io) \in \K(X_w)$, and therefore
\[
S_2 \, i_v^w(e'_\ka) \, i_u^w(e_\io) \, S_1 \in \K(X_{q y_u}, X_{p y_v}) \foral \io, \ka.
\]
Since $\K(X_{q y_u}, X_{p y_v})$ is norm closed the same holds for $S_2 S_1$.
\end{proof}

We come to the key technical result.
On each fibre of the Fock space, a word in the creation operators is the amplification of a single compact operator between two fibres, and the index of that fibre can be taken inside the constructible ideal attached to the word.

\begin{proposition} \label{P:fanormal}
Let $P$ be a unital finitely aligned subsemigroup of a discrete group $G$ and let $X$ be a compactly aligned saturated product system over $P$.
Let $p_1, q_1, \dots, p_n, q_n \in P$ and $\xi_{p_i} \in X_{p_i}, \xi_{q_i} \in X_{q_i}$ for $i = 1, \dots, n$, and consider 
\[
g := p_1^{-1} q_1 \cdots p_n^{-1} q_n \in G 
\ \text{ and } \
\Bx := q_n^{-1} p_n \dots q_1^{-1} p_1 P \in \J,
\]
and an element of the form
\[
b := \la_{p_1}(\xi_{p_1})^* \la_{q_1}(\xi_{q_1}) \cdots
\la_{p_n}(\xi_{p_n})^* \la_{q_n}(\xi_{q_n}) \in \T_\la(X).
\]
Then for every $r \in \Bx$ there exist $s \in \Bx$ and $z \in P$ with $r = s z$, and $S \in \K(X_s, X_{gs})$, such that
\[
b|_{X_r} = i^z_{gs, s}(S).
\]
\end{proposition}

\begin{proof}
Given $r \in \Bx$ we shall produce $s$, $z$ and $S$, and verify the identity.
Note that once $s \in \Bx$ and $r = s z$ are given, we have $gs, gr \in P$ by Proposition~\ref{P:mult} and $gr = (gs)z$ in $G$, so that $i^z_{gs,s}$ maps $\L(X_s, X_{gs})$ into $\L(X_r, X_{gr})$, which is where $b|_{X_r}$ lives.
We induct on $n$.

\smallskip
\noindent
{\it Base case $n = 1$.}
Here $b = \la_{p_1}(\xi_{p_1})^* \la_{q_1}(\xi_{q_1})$, $g = p_1^{-1} q_1$ and $\Bx = q_1^{-1} p_1 P = \{ r \in P \mid q_1 r \in p_1 P \}$.
Fix $r \in \Bx$, so that $q_1 r \in p_1 P \cap q_1 P$.
By Remark~\ref{R:famin} we may choose a minimal common right multiple $w$ of $p_1$ and $q_1$ with $q_1 r \in wP$, and we write
\[
w = p_1 y_1 = q_1 y_2
\qand
q_1 r = w z
\text{ for } y_1, y_2, z \in P.
\]
From $q_1 r = q_1 y_2 z$ and left cancellation we get $r = y_2 z$.
Set $s := y_2$, so that $r = s z$ and $g s = p_1^{-1} q_1 y_2 = p_1^{-1} w = y_1$.
Moreover, we have $q_1 s = w \in p_1 P$ and hence $s \in \Bx$.
Set
\[
S := i^{y_1}_{e_P, p_1}
\left( \left( M^{e_P, p_1}_{\xi_{p_1}} \right)^* \right)
\, i^{y_2}_{q_1, e_P} \left( M^{e_P, q_1}_{\xi_{q_1}} \right),
\]
which lies in $\K(X_{y_2}, X_{y_1}) = \K(X_s, X_{gs})$ by Lemma~\ref{L:famult} and Lemma~\ref{L:facomp}, the latter applied with $u := q_1$, $v := p_1$, $y_u := y_2$, $y_v := y_1$ and $p = q := e_P$.
By Lemma~\ref{L:facreate} and item (iii) of Lemma~\ref{L:faj} we have
\[
\la_{q_1}(\xi_{q_1})|_{X_r}
=
i^{y_2 z}_{q_1, e_P} \left( M^{e_P,q_1}_{\xi_{q_1}} \right)
=
i^{z}_{w, y_2} \left( i^{y_2}_{q_1, e_P}
\left( M^{e_P,q_1}_{\xi_{q_1}} \right) \right),
\]
since $r = y_2 z$ and $q_1 y_2 = w$, and likewise
\[
\la_{p_1}(\xi_{p_1})^*|_{X_{q_1 r}}
=
i^{y_1 z}_{e_P, p_1}
\left( \left( M^{e_P,p_1}_{\xi_{p_1}} \right)^* \right)
=
i^{z}_{y_1, w} \left( i^{y_1}_{e_P, p_1}
\left( \left( M^{e_P,p_1}_{\xi_{p_1}} \right)^* \right) \right),
\]
since $q_1 r = w z = p_1 y_1 z$.
Composing and using item (ii) of Lemma~\ref{L:faj} we obtain
\[
b|_{X_r} = i^{z}_{y_1, y_2}(S) = i^z_{gs, s}(S).
\]

\smallskip
\noindent
{\it Inductive step.}
Let $n \geq 2$ and write $b = b' c$, where $c := \la_{p_n}(\xi_{p_n})^* \la_{q_n}(\xi_{q_n})$ and $b'$ is the word with parameters $p_1, q_1, \dots, p_{n-1}, q_{n-1}$.
Write
\[
h := p_n^{-1} q_n, \quad \By := q_n^{-1} p_n P
\qand
g' := p_1^{-1} q_1 \cdots p_{n-1}^{-1} q_{n-1}, \quad
\Bx' := q_{n-1}^{-1} p_{n-1} \dots q_1^{-1} p_1 P,
\]
for the data attached to $c$ and to $b'$, so that $g = g' h$ and
\[
\Bx = q_n^{-1} p_n \Bx' = \{ r \in \By \mid h r \in \Bx' \}.
\]
Fix $r \in \Bx$.
By the base case applied to $c$ at $r$ there are $s_1 \in \By$ and $z_1 \in P$ with $r = s_1 z_1$, and $S_1 \in \K(X_{s_1}, X_{u})$ with $u := h s_1$, such that $c|_{X_r} = i^{z_1}_{u, s_1}(S_1)$.
By the inductive hypothesis applied to $b'$ at $h r \in \Bx'$ there are $v \in \Bx'$ and $z' \in P$ with $h r = v z'$, and $S' \in \K(X_v, X_{g'v})$, such that $b'|_{X_{hr}} = i^{z'}_{g'v, v}(S')$.

Since $h r = u z_1 = v z'$ we have $h r \in u P \cap v P$, and by Remark~\ref{R:famin} we may choose a minimal common right multiple $w$ of $u$ and $v$ with $hr \in wP$.
Write $w = u y_u = v y_v$ and $h r = w z$ with $y_u, y_v, z \in P$.
From $u z_1 = u y_u z$ and left cancellation we get $z_1 = y_u z$, and likewise $z' = y_v z$.
Set
\[
s := s_1 y_u
\qand
S := i^{y_v}_{g'v, v}(S') \; i^{y_u}_{u, s_1}(S_1),
\]
so that 
\[
r = s_1 y_u z = s z
\AND 
g s = g' h s_1 y_u = g' u y_u = g' w = g' v y_v,
\]
while $S \in \K(X_{s}, X_{gs})$ by Lemma~\ref{L:facomp} applied with $p := g'v$ and $q := s_1$.
By items (iii) and (ii) of Lemma~\ref{L:faj} we obtain
\begin{align*}
b|_{X_r}
&=
b'|_{X_{hr}} \circ c|_{X_r}
=
i^{z'}_{g'v, v}(S')
i^{z_1}_{u, s_1}(S_1)\\
&=
i^{y_v z}_{g'v, v}(S')
i^{y_u z}_{u, s_1}(S_1)
=
i^{z}_{g'w, w}\left( i^{y_v}_{g'v, v}(S') \right)
i^{z}_{w, s}\left( i^{y_u}_{u, s_1}(S_1) \right)
=
i^{z}_{gs, s}(S).
\end{align*}
Finally $s = s_1 y_u \in \By$, since $s_1 \in \By$ and $\By$ is a right ideal, while $h s = u y_u = w \in v P \subseteq \Bx'$, since $v \in \Bx'$ and $\Bx'$ is a right ideal.
Hence $s \in \Bx$ and the induction is complete.
\end{proof}

In order to pass from Proposition~\ref{P:fanormal} to the $\bo{K}$-cores we make use of Lemma~\ref{L:faunits}.

\begin{corollary} \label{C:facompact}
Let $P$ be a unital finitely aligned subsemigroup of a discrete group $G$ and let $X$ be a compactly aligned saturated product system over $P$.
Let $\mt \neq \Bx \in \J$ with irredundant generating family $\{s_1, \dots, s_m\}$.
Then
\[
\ev_{s_i} \left( \bo{K}_{\Bx, \la_\ast} \right) \subseteq \K(X_{s_i})
\foral i = 1, \dots, m.
\]
\end{corollary}

\begin{proof}
Fix $i \in \{1, \dots, m\}$.
By Proposition~\ref{P:fanoeps} it suffices to consider an element $b\in \bo{K}_{\Bx,\la_\ast}$ of the form 
\[
b:=\la_{p_1}(\xi_{p_1})^*\la_{q_1}(\xi_{q_1}) \cdots \la_{p_n}(\xi_{p_n})^*\la_{q_n}(\xi_{q_n}) \text{ where } n\in \bN, p_j,q_j \in P \foral j=1,\dots,n,
\]
such that $p_1^{-1} q_1\cdots p_n^{-1} q_n= e_P$ and $q_n^{-1}p_n\dots q_1^{-1}p_1P=\Bx$, since the linear span of such elements is dense in $\bo{K}_{\Bx, \la_\ast}$.
Proposition~\ref{P:fanormal} applied at $s_i \in \Bx$ provides $s \in \Bx$ and $z \in P$ with $s_i = s z$, and $S \in \K(X_s)$, such that $\ev_{s_i}(b) = i^z_{s,s}(S)$.
Since $s_i \in s P$ and $s \in \Bx$, Lemma~\ref{L:famax} gives $z \in P^*$, and therefore
\[
\ev_{s_i}(b) = i^z_{s,s}(S) \in \K(X_{sz}) = \K(X_{s_i}),
\]
by item (ii) of Lemma~\ref{L:faunits}.
The claim follows, since $\ev_{s_i}$ is contractive and $\K(X_{s_i})$ is closed.
\end{proof}

\begin{remark} \label{R:faneed}
Compact alignment is used in an essential way in Proposition~\ref{P:fanormal}, as the conclusion of Corollary~\ref{C:facompact} fails without this assumption.

For an example, consider the finitely aligned semigroup $P = \{0, 2, 3, 4, \dots\} \subseteq \bZ$ and let $H$ be an infinite-dimensional Hilbert space.
We may regard $Y:=H$ as a non-degenerate C*-correspondence over $A:=\bC$.
Setting $X_p := Y^{\otimes p}$ for $p \in P$ defines a saturated product system over $P$ with connecting maps $u_{p,q}$ given by the associativity of the stabilised tensor product.
Take 
\[
\Bx := (2+P) \cap (3+P) = \{5,6,7,\dots\},
\]
so that $\{5,6\}$ is an irredundant generating family for $\Bx$ by Example~\ref{Ex:fa}, and consider
\[
b
:=
\la^{(2)}(\theta^{X_2}_{\xi, \eta}) \,
\la^{(3)}(\theta^{X_3}_{\ze, \om})
\in \bo{K}_{\Bx, \la_\ast}
\textup{ for }
\xi, \eta \in X_2, \ \ze, \om \in X_3.
\]
Since $5 = 2 + 3 = 3 + 2$, Lemma~\ref{L:fadiag} and Lemma~\ref{L:facptfibre} give
\[
\ev_5(b) =\ev_5(\la^{(2)}(\theta^{X_2}_{\xi, \eta})) \ev_5(\la^{(3)}(\theta^{X_3}_{\ze, \om}))= i_2^5 \left( \theta^{X_2}_{\xi,\eta} \right) \, i_3^5 \left( \theta^{X_3}_{\ze,\om} \right).
\]
Consider $\ze = \ze_1 \otimes \ze_2$ with $\ze_1 \in Y^{\otimes 2}$ and $\ze_2 \in Y$, and set $\mu := \xi \sca{\eta, \ze_1} \otimes \ze_2 \in X_3$.
We will show that
\begin{equation}\label{Eq:comp5}
\ev_5(b) \nu = i_3^5 \left( \theta^{X_3}_{\mu, \om} \right) \nu
\foral \nu \in X_5.
\end{equation}
For $\mu \neq 0$ (for example by choosing $\xi=\eta=\ze_1 \neq 0$ and $\ze_2 \neq 0$) and $\om \neq 0$, we will get 
\[
\ev_5(b) 
= 
i_3^5 \left( \theta^{X_3}_{\mu, \om} \right)
=
\theta^{X_3}_{\mu, \om} \otimes \id_{Y^{\otimes 2}}
\notin \K(X_5),
\]
as $Y^{\otimes 2}$ is infinite-dimensional.
Hence the conclusion of Corollary~\ref{C:facompact} fails for this product system.
Moreover we will get that $X$ is not compactly aligned since
\[
i_2^5(\K(X_2)) \, i_3^5(\K(X_3)) \not\subseteq \K(X_5).
\]

For \eqref{Eq:comp5}, it suffices to consider $\nu = \nu_1 \otimes \nu_2$ with $\nu_1 \in X_3$ and $\nu_2 \in X_2$.
In this case we have
\begin{align*}
\ev_5(b) \nu
& =
i_2^5 \left( \theta^{X_2}_{\xi,\eta} \right) \, i_3^5 \left( \theta^{X_3}_{\ze,\om} \right) (\nu_1 \otimes \nu_2)
=
i_2^5 \left( \theta^{X_2}_{\xi,\eta} \right) \left(\ze_1 \otimes (\ze_2 \sca{\om,\nu_1}) \otimes \nu_2\right) \\
& =
(\xi \sca{\eta, \ze_1}) \otimes (\ze_2 \sca{\om,\nu_1}) \otimes \nu_2
=
(\mu \sca{\om, \nu_1}) \otimes \nu_2
=
i_3^5 \left( \theta^{X_3}_{\mu, \om} \right) \nu,
\end{align*}
as required.
\end{remark}

\subsection{Exactness}

We now provide the relevant results on exactness.

\begin{proposition} \label{P:famain}
Let $P$ be a unital finitely aligned subsemigroup of a discrete group $G$ and let $X$ be a compactly aligned saturated product system over $P$ with coefficient C*-algebra $A$.
If $A$ is exact, then $\bo{K}_{\Bx, \la_\ast}$ is exact for every $\Bx \in \J$.
\end{proposition}

\begin{proof}
For $\Bx = \mt$ we have $\bo{K}_{\mt, \la_\ast} = (0)$ by \eqref{Eq:faKempty}, and there is nothing to prove.
So consider $\mt \neq \Bx \in \J$.
By Proposition~\ref{P:fajfg} we may fix an irredundant generating family $\{p_1, \dots, p_m\}$ for $\Bx$.
By Proposition~\ref{P:fafibres} the map 
\[
\Phi_{\Bx} \colon \bo{K}_{\Bx, \la_\ast} \to \bigoplus_{i=1}^m \L(X_{p_i});
\
b \mapsto \left( \ev_{p_1}(b), \dots, \ev_{p_m}(b) \right),
\]
is an injective $*$-homomorphism, and by Corollary~\ref{C:facompact} we obtain that its range is actually contained in $\bigoplus_{i=1}^m \K(X_{p_i})$.
Hence $\bo{K}_{\Bx, \la_\ast}$ is $*$-isomorphic to a C*-subalgebra of $\bigoplus_{i=1}^m \K(X_{p_i})$.
Each $\K(X_{p_i})$ is exact by \cite[Proposition B.7]{Kat04}, a finite direct sum of exact C*-algebras is exact, and exactness passes to C*-subalgebras.
Therefore $\bo{K}_{\Bx, \la_\ast}$ is exact.
\end{proof}

\begin{theorem} \label{T:faexact}
Let $P$ be a unital finitely aligned subsemigroup of an exact discrete group $G$ and let $X$ be a compactly aligned saturated product system over $P$ with coefficient C*-algebra $A$.
Then the following are equivalent:
\begin{enumerate}
\item $A$ is exact.
\item $\T_\la(X)$ is exact.
\end{enumerate}
If in addition $G$ is amenable, then any of the items above is equivalent to the following:
\begin{enumerate}[resume]
\item $\ca(t)$ is exact for an injective Fock covariant representation $t$ of $X$.
\end{enumerate}
\end{theorem}

\begin{proof}
\noindent
[(i) $\Rightarrow$ (ii)].
Proposition~\ref{P:famain} implies that $\bo{K}_{\Bx, \la_\ast}$ is exact for every $\Bx \in \J$ and hence $\T_\la(X)$ is exact by applying Corollary~\ref{C:Excores-3}.

\smallskip
\noindent
[(ii) $\Rightarrow$ (i)].
This is immediate as $\la$ is injective and exactness passes to C*-subalgebras.

\smallskip
\noindent
[(ii) $\Leftrightarrow$ (iii)].
This is Corollary~\ref{C:Excores-2}.
\end{proof}

\subsection{Compactly aligned product systems over strongly finitely aligned semigroups}

We next consider the question of nuclearity.
This will apply to a subclass of finitely aligned semigroups considered also in \cite{ES18}.

\begin{definition}
A unital subsemigroup $P$ of a discrete group $G$ is called \emph{strongly finitely aligned} if for every $p,q \in P$ the constructible ideal $pP\cap qP$ can be written as a disjoint union of finitely many principal right ideals of $P$.
\end{definition}

In particular a strongly finitely aligned semigroup is finitely aligned.
The next proposition shows that, as in the case of finite alignment, this stronger condition of disjoint unions passes to the constructible ideals.

\begin{proposition} \label{P:sfa}
Let $P$ be a unital subsemigroup of a discrete group $G$.
Then the following are equivalent:
\begin{enumerate}
\item $P$ is strongly finitely aligned.
\item Every $\Bx\in \J$ is a disjoint union of finitely many principal right ideals.
\end{enumerate}
\end{proposition}

\begin{proof}
\noindent
[(ii) $\Rightarrow$ (i)].
This is immediate, since $pP \cap qP \in \J$.

\smallskip
\noindent
[(i) $\Rightarrow$ (ii)].
If $\Bx=\mt$ the conclusion holds trivially, so fix a non-empty $\Bx \in \J$.
In this case $\Bx=q_n^{-1}p_n \dots q_1^{-1}p_1 P$ for some $n \geq 1$ and $p_i, q_i \in P$.
We proceed by induction on $n \geq 1$.
If $n=1$, then $\Bx=q_1^{-1}p_1 P=q_1^{-1}(p_1P \cap q_1P)$.
Since $\Bx \neq \mt$ we get that $p_1P \cap q_1P \neq \mt$ and hence item (i) implies that
\[
p_1P \cap q_1P = \bigcup_{j=1}^{m} s_jP \text{ for some } m \in \bN, s_1, \dots, s_m \in P, \textup{ such that } s_i P \cap s_j P = \mt \text{ for } i \neq j.
\]
Since $P$ is unital we have $s_j \in q_1P$, say $s_j = q_1 t_j$ with $t_j \in P$.
It follows that
\[
t_iP \cap t_jP = q_1^{-1}(s_iP \cap s_j P) = \mt, \textup{ when } i \neq j,
\]
and moreover we get
\[
q_1^{-1}(p_1P) = q_1^{-1}(p_1P \cap q_1 P) = \bigcup_{j=1}^{m} t_jP.
\]
Now set $\Bx_0 := P$ and $\Bx_k := q_k^{-1}p_k \Bx_{k-1}$ for $k = 1, \dots, n$, so that $\Bx = \Bx_n$.
Suppose that
$\Bx_{k-1} = \bigcup_{i=1}^{m} r_iP$ for some $r_1, \dots, r_m \in P$ where the union is disjoint.
Left cancellation gives $p_k\Bx_{k-1} = \bigcup_{i=1}^{m} p_kr_iP$.
Since preimages of disjoint sets are disjoint, we obtain
\[
\Bx_k = q_k^{-1}(p_k\Bx_{k-1}) = \bigcup_{i=1}^{m} q_k^{-1}(p_kr_iP),
\]
where the union is disjoint.
The case $n=1$ applied to $q_k$ and $p_kr_i$ yields that each $q_k^{-1}(p_kr_iP)$ is a disjoint union of finitely many principal right ideals.
Hence so is $\Bx_k$, which completes the induction.
\end{proof}

\begin{remark}
If $P$ is a semigroup such that any two principal right ideals of $P$ intersect, then a disjoint decomposition of $\mt \neq \Bx \in \J$ into principal right ideals must consist of a single ideal.
Therefore in this case $P$ is strongly finitely aligned if and only if it is a right LCM semigroup.
This occurs for example when $P$ is commutative, or right Ore.

Another example is provided by semigroups with an independent set of constructible ideals.
Recall that $\J$ is called \emph{independent} if $\Bx = \bigcup_{i=1}^n \Bx_i$ with $\Bx, \Bx_1, \dots, \Bx_n \in \J$ implies $\Bx = \Bx_i$ for some $i$, see \cite{Li12}.
If $P$ is strongly finitely aligned and $\J$ is independent, then $P$ is again a right LCM semigroup.
Indeed, writing $\Bx \in \J$ as a disjoint union $\bigcup_{i=1}^n p_i P$ with $n \geq 2$ would give $\Bx = p_i P$ for some $i$ by independence, contradicting the disjointness (as $p_j \in \Bx \setminus p_i P$ for $j \neq i$).
Hence every $\Bx \in \J$ is empty or principal.
\end{remark}

\begin{example}\label{E:BS}
For an example of a strongly finitely aligned semigroup that is not a right LCM semigroup and is embeddable in an amenable discrete group consider first the Baumslag--Solitar group $\mathrm{BS}(1,2)$.
This group has an affine realisation generated by
\[
a(s)=2s
\qand
b(s)=s+1,
\]
from which it follows that $\mathrm{BS}(1,2)\simeq \bZ[1/2]\rtimes\bZ$, and thus it is amenable.

Moreover, it is well known that the semigroup homomorphism
\[
\rho \colon \bF_2^+ \to \mathrm{BS}(1,2); \ x \mapsto a, y \mapsto ba,
\]
is injective.
Indeed, set $t_0 := \rho(x) = a$ and $t_1 := \rho(y) = ba$, so that $t_0(s) = 2s$ and $t_1(s) = 2s+1$, and note that $\rho(w_1 \cdots w_n) = \rho(w_1) \circ \cdots \circ \rho(w_n)$ for $w_i \in \{x,y\}$.
For $\eps_1,\ldots,\eps_n\in\{0,1\}$, a product of these elements in the affine realisation has the form
\[
t_{\eps_1}\cdots t_{\eps_n}(s)
=2^ns+k,
\]
where $k = \sum_{i=1}^{n} \eps_i 2^{i-1}$, as follows by induction on $n$.
Since $0 \leq k < 2^n$, the uniqueness of the binary expansion shows that the affine map $s \mapsto 2^n s + k$ determines both $n$ and the word $\eps_1\cdots\eps_n$.
Hence $\bF_2^+ \hookrightarrow \mathrm{BS}(1,2)$.

It remains to identify a strongly finitely aligned subsemigroup of $\bF_2^+$ that is not a right LCM semigroup.
Towards this end, let $P$ be the unital subsemigroup of $\bF_2^+$ generated by $\{x, xy, yx, yyx\}$.
We record an observation that we will use repeatedly.
The generating set is not a prefix code, as $x$ is a prefix of $xy$, and so factorisations in $P$ need not be unique.
However, if a non-empty $w \in P$ begins with $y$, then the leading generator of any factorisation of $w$ is $yx$ or $yyx$, and it is determined by the second letter of $w$; consequently
\begin{equation} \label{Eq:BSlead}
w \in P \text{ begins with } y
\quad \Longrightarrow \quad
w \in yxP \cup yyxP,
\end{equation}
and no element of $P$ begins with $yyy$.
In particular, since $yx \in P$, for every $u \in \bF_2^+$ we have
\begin{equation} \label{Eq:BScancel}
yxu \in P
\quad \Longleftrightarrow \quad
u \in P,
\end{equation}
where the forward implication follows from \eqref{Eq:BSlead} and left cancellation.

First we see that $P$ is finitely aligned.
For $w \in \bF_2^+$ set
\[
P_w := \{ v \in P \mid w v \in P \}.
\]
Let $p,q\in P$.
We say that $p$ and $q$ are \emph{prefix-comparable} if there is a $w \in \bF_2^+$ such that $q=pw$ or $p=qw$.
Since $P$ is contained in the free semigroup $\bF_2^+$, if $p$ and $q$ are prefix-incomparable, then no word can have both $p$ and $q$ as prefixes, and hence $pP\cap qP=\mt$.
Suppose that $p$ and $q$ are prefix-comparable and, without loss of generality, that $q = pw$ for the unique $w \in \bF_2^+$ with this property.
An element of $qP$ has the form $qv = pwv$ with $v \in P$, and it lies in $pP$ if and only if $wv \in P$.
Therefore
\begin{equation} \label{Eq:BSint}
pP \cap qP = q P_w,
\end{equation}
so that $pP \cap qP$ is determined by $q$ and by the word $w$ alone.
We claim that
\[
P_w = P
\quad \OR \quad
P_w = \mt
\quad \OR \quad
P_w = xP \cup yxP.
\]
In particular, the conclusion does not depend on the factorisations of $p$ and $q$ in $P$, which enter only in locating $w$.
Indeed, fix a factorisation $q = g_1 \cdots g_n$ for some $n \in \bN$ and $g_i \in \{x, xy, yx, yyx\}$.
We have the following cases.

\smallskip
\noindent
{\it Case 1.}
If $p = g_1 \cdots g_m$ for some $m \leq n$, then $w = g_{m+1} \cdots g_n \in P$ and thus $P_w = P$.

\smallskip
\noindent
{\it Case 2.} Otherwise $p = g_1 \cdots g_m \al$ for some $m < n$, where $g_{m+1} = \al \be$ with $\al, \be \in \bF_2^+$ non-empty, and $w = \be g_{m+2} \cdots g_n$.
In particular, we get that $\be \in \{x, y, yx\}$.

\begin{list}{}{%
  \setlength{\leftmargin}{.5em}%
  \setlength{\itemindent}{0pt}%
  \setlength{\listparindent}{0pt}%
  \setlength{\parsep}{\smallskipamount}%
  \setlength{\itemsep}{\smallskipamount}%
  \setlength{\topsep}{\smallskipamount}}
\item[--] {\it Case 2a.}
If $\be \in \{x, yx\}$, then $w \in P$ and thus $P_w = P$.

\item[--] {\it Case 2b.}
If $\be = y$, then $w = y z_1 \cdots z_\ell$, for $\ell := n-m-1$ and $z_i := g_{m+1+i}$.
We show by induction on $\ell \geq 0$ that $P_{y z_1 \cdots z_\ell}$ is one of $P$, $\mt$ or $xP \cup yxP$, for all $z_1, \dots, z_\ell \in \{x, xy, yx, yyx\}$.

For $\ell = 0$ we have $w = y$.
If $v \in P_y$, then $yv \in P$ begins with $y$, and so $yv \in yxP \cup yyxP$ by \eqref{Eq:BSlead}; cancelling $y$ gives $v \in xP \cup yxP$.
Conversely $y \cdot xP = yxP \subseteq P$ and $y \cdot yxP = yyxP \subseteq P$.
Hence $P_y = xP \cup yxP$.

For $\ell \geq 1$ we distinguish according to $z_1$.
If $z_1 \in \{x, yx\}$, then $y z_1 \in \{yx, yyx\}$ is a generator, so $w = (y z_1) z_2 \cdots z_\ell \in P$ and thus $P_w = P$.
If $z_1 = yyx$, then $w$ begins with $yyy$, and so does $wv$ for every $v \in P$; since no element of $P$ begins with $yyy$, we get $P_w = \mt$.
If $z_1 = xy$, then $w = yx \cdot w'$ for $w' := y z_2 \cdots z_\ell$, and \eqref{Eq:BScancel} gives
\[
v \in P_w
\iff
w' v \in P
\iff
v \in P_{w'},
\]
that is $P_w = P_{w'}$, and the inductive hypothesis applies to $w'$.
\end{list}

For strong finite alignment, it remains only to check that the unions $qxP \cup qyxP$ are disjoint.
Indeed, if $qxu=qyxv$ for some $u,v\in P$, cancellation of $q$ in the free semigroup gives $xu=yxv$, which is impossible because the word on the left begins with $x$ and the word on the right begins with $y$.
Thus $qxP\cap qyxP=\mt$, as required.

In order to see that $P$ is not a right LCM semigroup, we apply the above to $p := x$ and $q := xy$, so that $w = y$ and
\[
xP\cap xyP
=
xyxP \cup xyyxP.
\]
This intersection is not principal.
Indeed, if it equalled $zP$ for some $z \in P$, then $z \in zP$ would give $z \in xyxP \cup xyyxP$ and hence $|z| \geq 3$, while $xyx, xyyx \in zP$ would force $z$ to be a prefix of both $xyx$ and $xyyx$, whose longest common prefix is $xy$; thus $|z| \leq 2$, a contradiction.
Hence the semigroup $P$ is not a right LCM semigroup, as required.
\end{example}

\begin{proposition} \label{P:facore}
Let $P$ be a unital finitely aligned subsemigroup of a discrete group $G$ and let $X$ be a compactly aligned saturated product system over $P$.
Let $\mt \neq \Bx \in \J$ and suppose that $\Bx = \bigcup_{i=1}^m p_i P$ with $p_i P \cap p_j P = \mt$ for all $i \neq j$.
Then
\[
\bo{K}_{\Bx, \la_\ast}
\subseteq
\sum_{i=1}^m \la^{(p_i)}(\K(X_{p_i})).
\]
If in addition $t$ is a Fock covariant representation of $X$, then
\[
\bo{K}_{\Bx, t_\ast}
\subseteq
\sum_{i=1}^m t^{(p_i)}(\K(X_{p_i})).
\]
\end{proposition}

\begin{proof}
Disjointness of the union implies that the family $\{p_1, \dots, p_m\}$ is irredundant.
Indeed, if $p_i P \subseteq p_j P$ for some $i,j$, then $p_i \in p_i P \cap p_j P $ and thus $i=j$.
Hence Corollary~\ref{C:facompact} yields
\[
k_i := \ev_{p_i}(b) \in \K(X_{p_i})
\foral b \in \bo{K}_{\Bx, \la_\ast}, \ i = 1, \dots, m.
\]

Fix $b \in \bo{K}_{\Bx, \la_\ast}$ and set $c := \sum_{i=1}^m \la^{(p_i)}(k_i)$.
By item (i) of Lemma~\ref{L:fadiag} and Lemma~\ref{L:facptfibre}, both $b$ and $c$ leave every fibre $X_r$ invariant. Since $\F X = \sum^\oplus_{r \in P} X_r$ it therefore suffices to show that $\ev_r(b) = \ev_r(c)$ for every $r \in P$.

If $r \notin \Bx$, then $\ev_r(b) = 0$ by item (i) of Lemma~\ref{L:fadiag}, while $r \notin p_i P$ for every $i$ and hence $\ev_r(c) = 0$ by Lemma~\ref{L:facptfibre}.
If $r \in \Bx$, then by disjointness there is exactly one index $i$ with $r \in p_i P$.
Writing $r = p_i z$ with $z \in P$ we obtain
\[
\ev_r(c)
=
i_{p_i}^{p_i z} (k_i)
=
i_{p_i}^{p_i z} \left( \ev_{p_i}(b) \right)
=
\ev_{p_i z}(b)
=
\ev_r(b),
\]
where the first equality follows from Lemma~\ref{L:facptfibre} together with $\ev_r(\la^{(p_j)}(k_j)) = 0$ for $j \neq i$, and the third equality is item (ii) of Lemma~\ref{L:fadiag}.
Hence $b = c$, which proves the first assertion.

For the second assertion, let $\dot{t} \colon [\F\C X]_{e_P} \to \ca(t)$ be the canonically induced $*$-homomorphism by the Fock covariance of $t$.
Since $\la^{(p)}(\K (X_p))\subseteq [\F\C X]_{e_P}$ and $\bo{K}_{\bo{x}, \la_\ast} \subseteq [\F\C X]_{e_P}$, it follows that $\dot{t} \circ \la^{(p)} = t^{(p)}$ and $\dot{t}(\bo{K}_{\Bx, \la_\ast}) = \bo{K}_{\Bx, t_\ast}$.
Applying $\dot{t}$ to the first assertion finishes the proof.
\end{proof}

We will also need the following proposition.

\begin{proposition} \label{P:fadirect}
Let $P$ be a unital subsemigroup of a discrete group $G$ and let $X$ be a saturated product system over $P$.
Let $p_1, \dots, p_m \in P$ satisfy $p_i P \cap p_j P = \mt$ for $i \neq j$.
Then
\[
\la^{(p_i)}(\K(X_{p_i})) \cdot \la^{(p_j)}(\K(X_{p_j})) = (0)
\foral i \neq j.
\]
If in addition $t$ is an equivariant Fock covariant representation of $X$, then 
\[
t^{(p_i)}(\K(X_{p_i})) \cdot t^{(p_j)}(\K(X_{p_j})) = (0) 
\foral i \neq j,
\]
and consequently
\[
\sum_{i=1}^m t^{(p_i)}(\K(X_{p_i}))
=
\bigoplus_{i=1}^m t^{(p_i)}(\K(X_{p_i}))
\]
is a C*-subalgebra of $[\ca(t)]_{e_P}$.
\end{proposition}

\begin{proof}
Fix $i \neq j$, and consider $k \in \K(X_{p_i})$ and $k' \in \K(X_{p_j})$.
Let also $\ze_r$ be in $X_r$ for some $r \in P$.
If $r \notin p_j P$, then Lemma~\ref{L:facptfibre} implies that $\la^{(p_j)}(k') \ze_r = 0$.
If $r \in p_j P$, then $\la^{(p_j)}(k') \ze_r \in X_r$ by the same lemma, while $r \notin p_i P$ by disjointness, and therefore $\la^{(p_i)}(k) \la^{(p_j)}(k') \ze_r = 0$.
Since $\F X = \sum^\oplus_{r \in P} X_r$ we conclude that $\la^{(p_i)}(k) \la^{(p_j)}(k') = 0$.

Now let $t$ be an equivariant Fock covariant representation of $X$ and let $\dot{t} \colon [\F\C X]_{e_P} \to \ca(t)$ be the canonically induced $*$-homomorphism so that $\dot{t} \circ \la^{(p)} = t^{(p)}$.
Applying $\dot{t}$ gives 
\[
t^{(p_i)}(k) t^{(p_j)}(k') = 0 \text{ for any }i \neq j.
\]
In particular the C*-algebras $t^{(p_i)}(\K(X_{p_i}))$ are mutually orthogonal C*-subalgebras of $\ca(t)$, and hence the sum $\sum_{i=1}^m t^{(p_i)}(\K(X_{p_i}))$ is a direct sum and a C*-subalgebra of $[\ca(t)]_{e_P}$ as each $t^{(p_i)}(\K(X_{p_i})) \subseteq [\ca(t)]_{e_P}$.
\end{proof}

\subsection{Nuclearity}

We now give the relevant results on nuclearity.

\begin{theorem} \label{T:fanuc}
Let $P$ be a unital strongly finitely aligned subsemigroup of an amenable discrete group $G$ and let $X$ be a compactly aligned saturated product system over $P$ with coefficient C*-algebra $A$.
Let $t$ be an injective equivariant Fock covariant representation of $X$.
Then the following are equivalent:
\begin{enumerate}
\item The embedding $A \hookrightarrow [\ca(t)]_{e_P}$ is nuclear.
\item The embedding $\bo{K}_{\Bx, t_\ast} \hookrightarrow [\ca(t)]_{e_P}$ is nuclear for every $\Bx \in \J$.
\item $\ca(t)$ is nuclear.
\end{enumerate}
\end{theorem}

\begin{proof}
\noindent 
[(ii) $\Leftrightarrow$ (iii)].
This follows from Theorem~\ref{T:Nuccores}.

\smallskip
\noindent
[(ii) $\Rightarrow$ (i)].
Taking $n = 1$, $p_1 = q_1 = e_P$ and $\eps = \eps' = 1$ in the definition of the $\bo{K}$-cores we get $p_1^{-1} q_1 = e_P$ and $q_1^{-1} p_1 P = P$, and therefore
\[
t_{e_P}(A) = \left[ t_{e_P}(A)^* t_{e_P}(A) \right] \subseteq
\bo{K}_{P, t_\ast}.
\]
Since $t$ is injective we may identify $A$ with $t_{e_P}(A)$, and the inclusion of $A$ in $[\ca(t)]_{e_P}$ is nuclear, since it factors as 
$A \hookrightarrow \bo{K}_{P, t_\ast} \hookrightarrow [\ca(t)]_{e_P}$.

\smallskip
\noindent
[(i) $\Rightarrow$ (ii)].
Fix $\Bx \in \J$.
If $\Bx=\mt$, then $\bo{K}_{\mt,t_\ast}=(0)$, and there is nothing to prove.
So suppose that $\Bx\neq \mt$.
Since $P$ is strongly finitely aligned, by Proposition~\ref{P:sfa} we may write
\[
\Bx = \bigcup_{i=1}^m p_i P \text{ with } p_i P \cap p_j P = \mt \foral i \neq j.
\]
Set $\C := \sum_{i=1}^m t^{(p_i)}(\K(X_{p_i}))$.
Since $t$ is injective, we get that $t^{(p)}$ is isometric for all $p\in P$, and hence nuclear embeddability of $t_{e_P}(A)$ in $[\ca(t)]_{e_P}$ implies nuclear embeddability of each $t^{(p_i)}(\K(X_{p_i}))$ in $[\ca(t)]_{e_P}$ by \cite[Proposition B.8]{Kat04}.
By Proposition~\ref{P:fadirect} the $t^{(p_i)}(\K(X_{p_i}))$ are mutually orthogonal C*-subalgebras of $[\ca(t)]_{e_P}$, and thus in turn $\C$ embeds nuclearly in $[\ca(t)]_{e_P}$.
By Proposition~\ref{P:facore} we have that $\bo{K}_{\Bx, t_\ast}$ is a C*-subalgebra of $\C$ and hence the inclusion map of $\bo{K}_{\Bx, t_\ast}$ in $[\ca(t)]_{e_P}$ is nuclear since it factors as
$\bo{K}_{\Bx, t_\ast} \hookrightarrow \C \hookrightarrow [\ca(t)]_{e_P}$.
\end{proof}

We obtain the following corollary.

\begin{corollary} \label{C:fanuc}
Let $P$ be a unital strongly finitely aligned subsemigroup of an amenable discrete group $G$ and let $X$ be a compactly aligned saturated product system over $P$ with coefficient C*-algebra $A$.
Then $\T_\la(X)$ is nuclear if and only if $A$ is nuclear.
\end{corollary}

\begin{proof}
\noindent
The Fock representation is injective, equivariant and Fock covariant.
If $A$ is nuclear, then $\la_{e_P}(A)$ is a nuclear C*-subalgebra of $[\T_\la(X)]_{e_P}$, hence nuclearly embedded in it, and Theorem~\ref{T:fanuc} applies.

For the converse, consider the completely contractive completely positive map
\[
\ev_{e_P} \colon \L(\F X) \to \L(X_{e_P}) = \L(A);\
c \mapsto Q_{e_P} c |_{X_{e_P}}.
\]
It is easy to check that $\ev_{e_P}$ restricts to a completely contractive completely positive map $\ev_{e_P}|_{\T_\la(X)} \colon \T_\la(X) \to A$ satisfying $\ev_{e_P}|_{\T_\la(X)} \circ \la_{e_P} = \id_A$.
If $\T_\la(X)$ is nuclear, then evidently $\la_{e_P} \colon A \to \T_\la(X)$ is nuclear.
Hence we have that $\id_A = \ev_{e_P}|_{\T_\la(X)} \circ \la_{e_P}$ is nuclear, as required.
\end{proof}

\section{Beyond compactly aligned product systems}\label{S:beyond}

We will see that compact alignment cannot be dropped from Theorem~\ref{T:faexact} and Theorem~\ref{T:fanuc}, even when $P = \bZ_+^2$ (hence a right LCM semigroup) and $A = \bC$.
In particular, exactness or nuclear embeddability of the coefficient C*-algebra is not enough to guarantee exactness or nuclearity for an injective equivariant Fock covariant representation without extra assumptions on the product system or the semigroup.

We first recall the following construction of a product system over $\bZ_+^2$.
This is an instance of the construction of product systems over right-angled Artin semigroups by Fowler--Sims \cite{FS02}; for the convenience of the reader we include some of the details.
 
Let $E$ and $F$ be C*-correspondences over a C*-algebra $A$ with the property that there is a unitary map
\[
u \colon F \otimes_A E \to E \otimes_A F.
\]
Consider the C*-correspondences over $A$ given by
\[
X_{(n,m)}
=
\begin{cases}
A & \textup{if } n = 0 \textup{ and } m = 0,\\
E^{\otimes n} & \textup{if } n \geq 1 \textup{ and } m = 0,\\
F^{\otimes m} & \textup{if } n = 0 \textup{ and } m \geq 1,\\
E^{\otimes n} \otimes_A F^{\otimes m} & \textup{if } n \geq 1 \textup{ and } m \geq 1.
\end{cases}
\]
The unitary maps $u_{p,q}$ for $p,q \in \bZ_+^2$ are given as follows.
If $p=(0,0)$ or $q=(0,0)$, then the maps come from the left or the right action by $A$.
In all other cases we construct the maps iteratively to satisfy the associativity.
For the first step we let
\[
u_{(1,0), (0,1)} := \id_{E} \otimes \id_{F}
\AND
u_{(0,1), (1,0)} := u.
\]
For the next steps, first we set
\[
u_{(n,0), (0,m)} := \id_{E^{\otimes n}} \otimes \id_{F^{\otimes m}}, \foral n, m \geq 1.
\]
Furthermore, for $n \geq 1$ and the map
\[
u_{(0,1), (n,0)} \colon F \otimes_A E^{\otimes n} \to E^{\otimes n} \otimes_A F,
\]
we set
\[
u_{(0,1), (n,0)} := (\id_{E^{\otimes (n-1)}} \otimes u) \circ \cdots \circ (\id_E \otimes u \otimes \id_{E^{\otimes (n-2)}}) \circ (u \otimes \id_{E^{\otimes (n-1)}}),
\]
where the $j$-th factor from the right is $\id_{E^{\otimes (j-1)}} \otimes u \otimes \id_{E^{\otimes (n-j)}}$ for $j = 1, \dots, n$; for $n = 1$ this reads $u_{(0,1),(1,0)} = u$.
Next for $n, m \geq 1$ and the map
\[
u_{(0,m),(n,0)} \colon F^{\otimes m} \otimes_A E^{\otimes n} \to E^{\otimes n} \otimes_A F^{\otimes m}
\]
we set
\begin{align*}
u_{(0,m),(n,0)}
:= {} &
(u_{(0,1), (n,0)} \otimes \id_{F^{\otimes (m-1)}})
\circ \cdots \circ
(\id_{F^{\otimes (m-1)}} \otimes u_{(0,1), (n,0)}),
\end{align*}
where the $j$-th factor from the right is $\id_{F^{\otimes (m-j)}} \otimes u_{(0,1),(n,0)} \otimes \id_{F^{\otimes (j-1)}}$ for $j = 1, \dots, m$; for $m = 1$ this reads $u_{(0,1),(n,0)}$.
Finally, for $n_1, n_2, m_1, m_2 \geq 1$ and the maps
\[
u_{(n_1,m_1), (n_2,m_2)} \colon (E^{\otimes n_1} \otimes_A F^{\otimes m_1}) \otimes_A (E^{\otimes n_2} \otimes_A F^{\otimes m_2}) \to E^{\otimes (n_1 + n_2)} \otimes_A F^{\otimes (m_1+m_2)}
\]
we set
\[
u_{(n_1,m_1), (n_2,m_2)} := \id_{E^{\otimes n_1}} \otimes u_{(0,m_1),(n_2,0)} \otimes \id_{F^{\otimes m_2}}.
\]
In order to deduce that the diagrams
\[
\xymatrix{
X_p \otimes_A X_q \otimes_A X_r \ar[rr]^{\id_{X_p} \otimes u_{q,r}} \ar[d]_{u_{p,q} \otimes \id_{X_r}} & & X_p \otimes_A X_{q+r} \ar[d]^{u_{p,q+r}} \\
X_{p+q} \otimes_A X_r \ar[rr]^{u_{p+q,r}} & & X_{p+q+r}
}
\]
are commutative for all $p,q,r \in \bZ_+^2$, it suffices to show this for $p,q,r \in \{(1,0), (0,1)\}$ as by associativity of $\otimes_A$ each diagram breaks into smaller ones on the generators of $\bZ_+^2$.
This follows by the definition of the unitaries.
 
\begin{example}\label{E:nonexact}
Below we obtain a product system over $\bZ_+^2$ (hence over a right LCM semigroup) that is not compactly aligned, with $A = \bC$ (hence exact), and we will have an $\Bx \in \J$ so that $\bo{K}_{\Bx, \la_\ast}$ contains a copy of $\ca(\bF_2)$ (hence a non-exact C*-subalgebra).
This will show that compact alignment cannot be dropped from Theorem~\ref{T:faexact}, and as nuclearity implies exactness, nor can it be dropped from Theorem~\ref{T:fanuc}.
In addition, this example shows that the assumption of the left actions containing the compacts is also necessary for \cite[Theorem 5.4, Theorem 5.6]{KLS25}.

Let $H$ be a separable infinite-dimensional Hilbert space with orthonormal basis $\{h_j\}_{j \in \bN}$.
Set $E = F := H$ and let $u \colon F \otimes E \to E \otimes F$ be any unitary; then $u$ is automatically a unitary for the C*-correspondences over $\bC$, and $E$ and $F$ are non-degenerate.
We consider the saturated product system $X$ over $\bZ_+^2$ with coefficients in $\bC$, such that
\[
X_{(1,0)} = E, \quad X_{(0,1)} = F, \quad X_{(1,1)} = E \otimes F, \quad
u_{(1,0),(0,1)} = \id_{E \otimes F}, \quad u_{(0,1),(1,0)} = u.
\]
For $e \in E$ and $f \in F$ consider 
\[
L_e \colon F \to E \otimes F;\ L_e f' = e \otimes f'
\qand 
M_f \colon E \to F \otimes E ;\ M_f e' = f \otimes e'.
\]
By Lemma~\ref{L:facptfibre} we have
\[
\ev_{(1,1)}\big(\la^{(1,0)}(k)\big) = k \otimes \id_F
\foral k \in \K(E),
\]
and
\[
\ev_{(1,1)}\big(\la^{(0,1)}(\theta^{F}_{f,f'})\big) 
= u (\theta^F_{f,f'} \otimes \id_E) u^*
= u M_f M_{f'}^* u^*
\foral f,f' \in F.
\]
 
Consider the constructible ideal
\[
\Bx := (1,1) + \bZ^2_+ = ((1,0)+\bZ^2_+) \cap ((0,1)+\bZ^2_+).
\]
Since $\la$ is a representation of $X$ we have
\[
\la_{(1,0)}(h_1)^*\la_{(1,0)}(h_1)=\la_{(0,0)}(\sca{h_1,h_1})=1,
\]
and thus the element
\[
Q := \la^{(1,0)}(\theta^E_{h_1,h_1}) = \la_{(1,0)}(h_1) \la_{(1,0)}(h_1)^* \in \bo{K}_{(1,0)+\bZ^2_+, \la_\ast}
\]
is a projection.
For $h_i, h_j \in F$ we define
\[
c_{ij} := Q \, \la^{(0,1)}(\theta^F_{h_i,h_j}) \, Q,
\]
for which we get
\[
c_{ij} = Q \, \la^{(0,1)}(\theta^F_{h_i,h_j}) \, Q \in \bo{K}_{(1,0)+\bZ^2_+, \la_\ast} \cdot \bo{K}_{(0,1)+\bZ^2_+, \la_\ast} \cdot \bo{K}_{(1,0)+\bZ^2_+, \la_\ast}\subseteq \bo{K}_{\Bx,\la_\ast}.
\]
Since $Q$ is a projection we have
\[
c_{ij}=Qc_{ij}Q \in Q \bo{K}_{\Bx,\la_\ast} Q.
\]
Moreover, since $Q \in \bo{K}_{(1,0)+\bZ^2_+, \la_\ast}$ and $((1,0)+\bZ^2_+) \cap \Bx = \Bx$, we have
\[
Q \bo{K}_{\Bx,\la_\ast} Q
\subseteq 
\bo{K}_{(1,0)+\bZ^2_+, \la_\ast} \cdot \bo{K}_{\Bx,\la_\ast} \cdot \bo{K}_{(1,0)+\bZ^2_+, \la_\ast} \subseteq \bo{K}_{\Bx,\la_\ast},
\]
and therefore
\[
Q \bo{K}_{\Bx,\la_\ast} Q = \{ c \in \bo{K}_{\Bx,\la_\ast} \mid c = QcQ \}.
\]
It follows that $Q \bo{K}_{\Bx,\la_\ast} Q$ is a C*-subalgebra of $\bo{K}_{\Bx,\la_\ast}$.
Since $\Bx = (1,1)+\bZ^2_+$ is singly generated, Proposition~\ref{P:fafibres} implies that the map
\[
\ev_{(1,1)}|_{\bo{K}_{\Bx,\la_\ast}} \colon \bo{K}_{\Bx,\la_\ast} \to \L(X_{(1,1)}) = \L(E \otimes F),
\]
is an injective $*$-homomorphism.
 
Henceforth we fix the isometry
\[
V := L_{h_1} \colon F \to E \otimes F;\ Vh = h_1 \otimes h.
\]
For its range projection $VV^*$ we have
\[
V V^* = \theta^E_{h_1,h_1} \otimes \id_F= \ev_{(1,1)}(Q).
\]
Since $\ev_{(1,1)}$ is a unital completely positive map on $\L(\F X)$ and the operators $Q$ and $\ev_{(1,1)}(Q)$ are projections, we get
\[
\ev_{(1,1)}(Q^* Q)= \ev_{(1,1)}(Q) = \ev_{(1,1)}(Q)^* \ev_{(1,1)}(Q),
\]
and
\[
\ev_{(1,1)}(Q Q^*)= \ev_{(1,1)}(Q) = \ev_{(1,1)}(Q) \ev_{(1,1)}(Q)^*.
\]
Hence $Q$ lies in the multiplicative domain of $\ev_{(1,1)}$.
Therefore, for every $c \in Q \bo{K}_{\Bx,\la_\ast} Q$ we have
\[
\ev_{(1,1)}(c) =\ev_{(1,1)}(QcQ) = \ev_{(1,1)}(Q)\, \ev_{(1,1)}(c)\, \ev_{(1,1)}(Q) = V V^* \ev_{(1,1)}(c)\, V V^*,
\]
and so $\ev_{(1,1)}(c) = V \Psi(c) V^*$, where
\[
\Psi \colon Q \bo{K}_{\Bx,\la_\ast} Q \to \B(F);\ 
c \mapsto V^* \ev_{(1,1)}(c) V.
\]
A direct calculation for $c \in Q \bo{K}_{\Bx,\la_\ast} Q$ gives
\begin{align*}
VV^* \ev_{(1,1)}(c) 
& = \ev_{(1,1)}(Q) \ev_{(1,1)}(c) 
= \ev_{(1,1)}(c) \\
& = \ev_{(1,1)}(c) \ev_{(1,1)}(Q) 
= \ev_{(1,1)}(c) V V^*,
\end{align*}
and so $\Psi$ is a $*$-homomorphism.
Finally, $\Psi$ is injective since $\ev_{(1,1)}(c) = V \Psi(c) V^*$, and $\ev_{(1,1)}$ is injective on $\bo{K}_{\Bx,\la_\ast} \supseteq Q \bo{K}_{\Bx,\la_\ast} Q$ by Proposition~\ref{P:fafibres}.
We conclude that $\Psi$ is an injective $*$-homomor\-phism which satisfies
\begin{equation}\label{Eq:psi}
\Psi(c_{ij}) 
= V^* \ev_{(1,1)}(c_{ij})V
= V^* u M_{h_i} M_{h_j}^* u^* V
= (V^* u M_{h_i})(V^* u M_{h_j})^*.
\end{equation}
 
The above hold for any separable infinite-dimensional Hilbert space $H$ and any unitary operator $u \colon F \otimes E \to E \otimes F$, which we now specify.
Consider the group C*-algebra $\ca(\bF_2)$ of the free group on two generators.
Since $\ca(\bF_2)$ is separable we may pick a separable infinite-dimensional Hilbert space $H$ so that we get a unital and faithful $*$-representation $\ca(\bF_2) \subseteq \B(H)$.
 
Let the unitaries $W_2, W_3 \in \B(H)$ be the images of the free generators under this representation, and put $W_1 := I_H$ so that 
\[
\ca(W_2, W_3) = \ca(W_1,W_2, W_3) \simeq \ca(\bF_2).
\]
Set $E = F := H$ and define
\[
R := \tfrac{1}{\sqrt 3} \sum_{j=1}^3 W_j M_{h_j}^* \colon F \otimes E \to F.
\]
Since $M_{h_i}^* M_{h_j} = \de_{i,j} \id_E$, we get $RR^* = \id_F$ and 
\[
R M_{h_j} = W_j/\sqrt{3}
\textup{ for }
j = 1,2,3.
\]
As $R$ is a co-isometry we have 
\[
\ker R = (1 - R^*R)(F \otimes E).
\]
The spaces $\ker R$ (which contains the ranges of $M_{h_j}$ for $j \geq 4$) and $(V F)^\perp$ (which contains the ranges of $L_{h_j}$ for $j \geq 2$) are separable and infinite-dimensional.
Choose a unitary 
\[
Z_0 \colon \ker R \to (VF)^\perp
\]
and extend it to a partial isometry
\[
Z \colon F \otimes E \to E \otimes F;\ Z|_{(\ker R)^\perp} = 0 \textup{ and } Z|_{\ker R} = Z_0.
\]
Hence $Z$ satisfies
\begin{align*}
Z^*Z = P_{\ker R} = \id_{F \otimes E} - R^*R & \qand
RZ^* = R Z^*ZZ^* = R (\id_{F \otimes E} - R^*R) Z^* = 0, \\
ZZ^* = P_{(VF)^\perp} = \id_{E \otimes F} - VV^* & \qand
Z^*V = Z^* Z Z^* V = Z^* (\id_{E \otimes F} - VV^*) V = 0.
\end{align*}
 
We now define the operator
\[
u := V R + Z \colon F \otimes E \to E \otimes F.
\]
On the one hand we have
\begin{align*}
u^*u 
& = (R^*V^* + Z^*)(VR + Z) 
= R^*V^*VR + R^*V^*Z + Z^*VR + Z^*Z \\
& = R^*R + 0 + 0 + (\id_{F \otimes E} - R^*R) = \id_{F \otimes E}.
\end{align*}
On the other hand we have
\begin{align*}
uu^*
& = (VR + Z) (R^*V^* + Z^*) 
= VRR^*V^* + VRZ^* + ZR^*V^* + ZZ^* \\
& = 
VV^* + 0 + 0 + (\id_{E \otimes F} - VV^*) = \id_{E \otimes F},
\end{align*}
and thus $u$ is a unitary.
 
Moreover we have $V^*u = V^*VR + V^*Z = R$, and applying this to \eqref{Eq:psi} we get
\[
\Psi(c_{ij})
= (V^* u M_{h_i}) (V^* u M_{h_j})^*
= (R M_{h_i}) (R M_{h_j})^*.
\]
For $i,j \in \{1,2,3\}$ we obtain
\[
\Psi(c_{ij}) = \frac{1}{3} W_i W_j^*.
\]
Since $W_1 = I_H$ and $\Psi$ is linear, we conclude that
\[
\Psi(3 c_{21}) = W_2 \qand \Psi(3 c_{31}) = W_3.
\]
Since $\Psi$ is a $*$-homomorphism it has closed range, and hence 
\[
\Psi(Q \bo{K}_{\Bx,\la_\ast} Q) \supseteq \ca(W_2,W_3) \simeq \ca(\bF_2),
\]
which is not exact, see \cite{Was76, Was78, Was90}.
As $\Psi$ is injective and exactness passes to C*-subalgebras, $\bo{K}_{\Bx,\la_\ast}$ is not exact.
In particular, the Fock C*-algebra $\T_\la(X)$ is not exact, and hence also not nuclear, even though the coefficient C*-algebra $A=\bC$ is nuclear.
 
In accordance with Theorem~\ref{T:faexact} and Theorem~\ref{T:fanuc}, the product system $X$ is not compactly aligned.
The element $(1,1)$ is the unique minimal common right multiple of $(1,0)$ and $(0,1)$.
Since $u_{(1,0),(0,1)} = \id_{E \otimes F}$ and $u_{(0,1),(1,0)} = u$, we have
\[
i_{(1,0)}^{(1,1)}(\theta^E_{h_1,h_1}) = \theta^E_{h_1,h_1} \otimes \id_F = V V^*
\qand
i_{(0,1)}^{(1,1)}(\theta^F_{h_2,h_1}) = u M_{h_2} M_{h_1}^* u^*.
\]
Set 
\[
T := i_{(1,0)}^{(1,1)}(\theta^E_{h_1,h_1}) \, i_{(0,1)}^{(1,1)}(\theta^F_{h_2,h_1}).
\]
Using $V^* V = \id_F$, $V^* u = R$ and $R M_{h_j} = \frac{1}{\sqrt 3} W_j$, we obtain
\[
T V = V (R M_{h_2}) (R M_{h_1})^* = \tfrac{1}{3} V W_2 W_1^* = \tfrac{1}{3} V W_2.
\]
If $T \in \K(X_{(1,1)})$, then $T V \in \K(F, X_{(1,1)})$ and therefore
$(TV)^*(TV) \in \K(F)$.
However 
\[
(TV)^*(TV) = \frac{1}{9} W_2^* V^* V W_2 = \frac{1}{9} \id_F \notin \K(F),
\]
as $F$ is infinite-dimensional.
Hence $T \notin \K(X_{(1,1)})$, and so $X$ is not compactly aligned.
\end{example}

\subsection*{Acknowledgements}

The authors would like to thank Mihalis Anoussis and Aristides Katavolos for their comments and remarks on a draft of the manuscript.

The second author acknowledges that this research work was supported by the Hellenic Foundation for Research and Innovation (HFRI) under the 5th Call for HFRI PhD Fellowships (Fellowship Number: 19145).

This material is based upon work supported by the Swedish Research Council under grant no.\ 2021-06594 while the authors were in residence at Institut Mittag-Leffler in Djursholm, Sweden during the first semester of 2026.
The authors would like to thank the Institut Mittag-Leffler and the organisers of the program ``Operator Algebras and Quantum Information" for the hospitality.

\subsection*{Conflict of interest}

The authors declare that they have no conflict of interest.


\end{document}